\documentclass[11pt,a4paper,reqno]{amsart}

\usepackage{graphicx, amsmath, amsthm, amssymb, enumerate, esint, mathtools,xcolor}
\usepackage[shortlabels]{enumitem}
\usepackage{amsfonts,mathrsfs,color,multirow}
\usepackage{pgf,tikz}
\usetikzlibrary{arrows}
\usepackage[toc,page]{appendix}

\graphicspath{ {./ImagesForFigures/} }
\numberwithin{equation}{section}

\DeclareMathOperator{\Tr}{Tr}
\newcommand{\R}[1]{\mathbb{R}^{#1}}
\newcommand{\N}{\mathbb{N}}
\newcommand{\Q}{\mathbb{Q}}

\newcommand{\rot}{\text{rot}}

\newcommand{\twobytwo}[4]{\left(\begin{array}{cc}
#1 & #2  \\
#3 & #4 \end{array} \right)}
\newcommand{\fpd}[2]{\frac{\partial #1}{\partial #2}}

\numberwithin{equation}{section}

\theoremstyle{plain}
\newtheorem{theorem}{Theorem}[section]
\newtheorem{prop}[theorem]{Proposition}

\theoremstyle{definition}

\newtheorem{example}[theorem]{Example}

\theoremstyle{remark}
\newtheorem{remark}[theorem]{Remark}

\allowdisplaybreaks

\title[Linear Stability of Periodic Trajectories]{Linear Stability of Periodic Trajectories in Inverse Magnetic Billiards}
\author[S. Gasiorek]{Sean Gasiorek}
\email{sean.gasiorek@sydney.edu.au}
\address{School of Mathematics and Statistics, Carslaw Building F07, University of Sydney, NSW 2006, Australia
}


\begin{document}

\begin{abstract}
We study the stability of periodic trajectories of planar inverse magnetic billiards, a dynamical system whose trajectories are straight lines inside a connected planar domain $\Omega$ and circular arcs outside $\Omega$. Explicit examples are calculated in circles, ellipses, and the one parameter family of curves $x^{2k}+y^{2k}=1$. Comparisons are made to the linear stability of periodic billiard and magnetic billiard trajectories. 
\end{abstract}

\maketitle

\section{Introduction}\label{intro}

Mathematical billiards serves as a fundamental example of a dynamical system and has been studied extensively over the last century. Connecting geometry and dynamics, mathematical billiards concerns the motion of a free particle (the ``billiard ball") under inertia in a domain (the ``billiard table") which moves at constant speed and undergoes perfectly elastic collisions with the boundary of the table.  The collisions with the boundary follow the billiard reflection law ``angle of incidence equals angle of reflection", where the tangential component of the velocity is conserved while the normal component instantly changes sign.  See e.g. \cite{Bir, KozlovTr, Tab} for a survey. 

While mathematical billiards serves as a model of certain physical phenomena, such as wave fronts and geometric optics, magnetic variants of billiards, where the billiard ball is interpreted as a charged particle  moving under the influence of a magnetic field $\boldsymbol{B}$ which satisfies the reflection law at the boundary of the billiard table, provides an extension of these same ideas to various geometric settings  (e.g. \cite{BMS2020, Gutkin2001, TabMag}) and to problems in dynamics and mathematical physics (e.g. \cite{BK, Dullin1998, KP05, RB1985, Robnik1986}). 

The definition of the magnetic field $\boldsymbol{B}$ greatly affects the dynamics but are also informed by the problem which is to be solved. The study of charged-particle dynamics in piecewise-constant magnetic fields appears in a variety of settings, such as nano- and condensed-matter physics, semiconductor design, and quantum mechanics \cite{CP2012, KPC2005, KROC2008, VTCP}. Classical, semiclassical, and quantum approaches to this system are each addressed to a degree in compact or unbounded domains depending upon the applications of interest.  

The magnetic billiard of interest is that of \emph{inverse magnetic billiards}, following the naming by \cite{VTCP}, which has only been studied in detail recently \cite{G2019, G2021}. Given a connected domain $\Omega \subset \R{2}$, define a constant magnetic field $\boldsymbol{B}$ orthogonal to the plane which has strength 0 on $\Omega$ and strength $B \neq 0$ on its complement. The classical motion of a charged particle of charge $e$ and mass $m$ with constant speed $|v|$ throughout $\Omega$ and its complement are continuous curves which are circular arcs outside $\Omega$ and straight chords inside $\Omega$. The charged particle is subject to the Lorentz force outside $\Omega$, and the resulting motion is circular arcs of fixed Larmor radius $\mu=m|v|/|eB|$. We take $eB<0$ so the Larmor arcs are traversed in the anticlockwise direction. 

This paper is organized as follows. Section \ref{IMBIntro} constructs the inverse magnetic billiard map, its derivative, and relevant properties. Section \ref{LinStab2} establishes a linear stability criteria for 2-periodic trajectories in inverse magnetic billiards, provides explicit examples, and contrasts the stability criteria with the existing linear stability criteria for standard and magnetic billiards. In Section \ref{LinStab34}, we give examples of the linear stability of 3- and 4-periodic trajectories with symmetries in various domains.

\section{Properties of the Inverse Magnetic Billiard Map}
\label{IMBIntro}

We give a brief review of inverse magnetic billiards in a convex set $\Omega \subset \R{2}$ and note that additional details can be found in \cite{G2019,G2021}.

Suppose $\Omega \subset \R{2}$ is strictly convex and parametrize the boundary $\partial \Omega = \Gamma(s)$ by arc length, $s$ in the anticlockwise direction and let $L = |\partial\Omega|$. Provided $\Gamma(s)$ is sufficiently smooth, the convexity of $\Omega$ implies the radius of curvature $\rho(s) = 1/\kappa(s)$ of $\Gamma(s)$ satisfies $0 <  \rho_{min} \leq \rho(s) \leq \rho_{max} < \infty$. In the study of magnetic billiards, Robnik and Berry \cite{RB1985,Robnik1986} classified the dynamics based upon three \emph{curvature regimes}, depending upon the relative sizes of $\mu$, $\rho_{min}$, and $\rho_{max}$: 
$$\mu < \rho_{min}, \qquad \rho_{min} < \mu < \rho_{max}, \qquad \mu> \rho_{max}.$$ 

The motion of the charged particle (the ``billiard") in this setting can be described in terms of its two geometric components: its straight-line standard billiard component and its magnetic Larmor arc. Suppose a trajectory starts at a point $P_0 := \Gamma(s_0)$ with an initial velocity vector $v_0$ pointing to the interior of $\Omega$. This initial velocity vector makes an angle $\theta_0 \in (0, \pi)$ with the positively-oriented tangent vector to $\Gamma(s)$ at $P_0$, and as the billiard travels in a straight line inside $\Omega$, it will eventually meet $\Gamma(s)$ at a new point $P_1 := \Gamma(s_1)$. Let $\ell_1 = |P_0P_1|$ be the chord length and $\theta_1$ be the angle between the tangent vector to $\Gamma(s)$ at $P_1$ and the velocity vector $\overrightarrow{P_0P_1}$. This completes the billiard-like component of motion. 

Next, because $\partial\Omega$ acts as a permeable boundary between magnetic and non-magnetic regions, the motion starting at $P_1$ in the direction $v_0$ moves along a circular Larmor arc $\gamma$ of radius $\mu$ until intersecting $\Gamma(s)$ at a point $P_2 :=\Gamma(s_2)$ where the billiard re-enters $\Omega$. The velocity vector $v_2$ of the billiard at $P_2$ is the tangent vector to the Larmor circle at $P_2$ and makes an angle $\theta_2$ with the tangent vector to $\Gamma(s)$ at $P_2$. The re-entry point $P_2$ is well-defined, as the Larmor circle is tangent to the ray $\overrightarrow{P_0P_1}$ at $P_1$. Let $\ell_2 = |P_1P_2|$ be the chord length connecting the exit and re-entry points $P_1$ and $P_2$ of $\Omega$; let $\chi_{0,2}$ be the angle between $v_0$ and the ray $\overrightarrow{P_1P_2}$; let $\mathcal{S}_{0,2}$ be the area inside the Larmor arc $\gamma_{0,2}$ but outside $\Omega$; let $\kappa_i := \kappa(s_i)$ be the curvature of $\Gamma$ at $s_i$; and let $\mathcal{A}_{0,2}$ denote the area of the ``cap" between the chord $\ell_2$ and $\Gamma(s)$ that is inside the Larmor circle.  See figure \ref{IMBMap}. 

\begin{figure}[thp]
\includegraphics[width=0.6\textwidth]{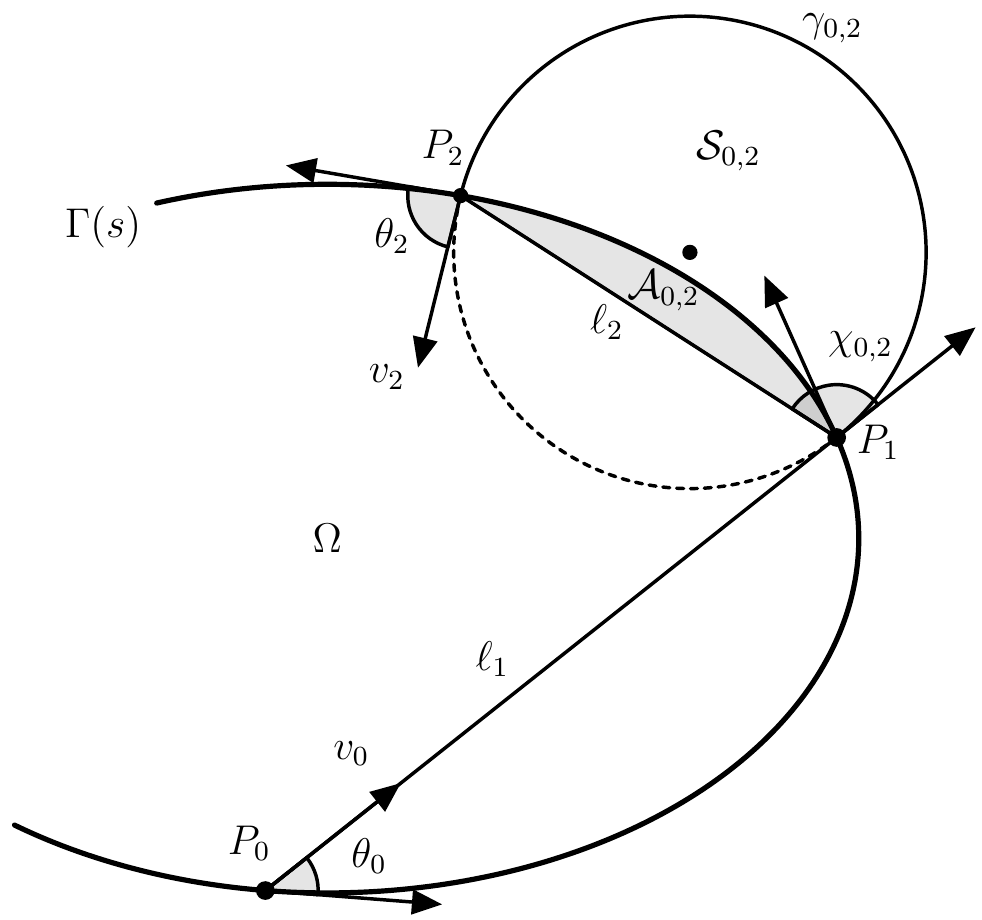} 
\caption{A labeled diagram of both components of motion of the inverse magnetic billiard trajectory.}
\label{IMBMap}
\end{figure}

\begin{remark} In general, we will keep the same numbering convention described above, so that the quantities $\theta_{2i}$, $\theta_{2i+1}$, $\theta_{2i+2}$, $\chi_{2i,2i+2}$, $\ell_{2i+1}$, $\ell_{2i+2}$, etc. are all associated with the trajectory from $P_{2i} \to P_{2i+2}$. 
\end{remark}

From the above definitions, we can encode the dynamics through the inverse magnetic billiard map $T$ on the open annulus $\R{}/L\mathbb{Z} \times (0,\pi) \approx [0,L) \times (0,\pi)$, which sends each successive reentry point and angle of the inward pointing velocity vector to that of the next reentry point and angle:
$$T: [0,L) \times (0,\pi) \to [0,L) \times (0,\pi), \qquad (s_i,\theta_i) \mapsto (s_{i+2},\theta_{i+2}).$$
By assuming that motion stops when $\theta = 0$ or $\pi$, we can extend the domain of this map to the closed annulus $\mathbb{A}:= [0, L) \times [0,\pi]$; that is, by continuity, $T(s,0) = (s,0)$ and $T(s,\pi) = (s,\pi)$.

The map $T$ preserves the symplectic area form $\sin(\theta) ds \wedge d\theta$ on the annulus $\mathbb{A}$. This informs the definition of area-preserving coordinates, $u_i := -\cos(\theta_i)$.

\begin{prop}[\cite{G2019,G2021}]
Suppose $\partial\Omega$ is of class $C^k$. Then the inverse magnetic billiard map $T$ is $C^{k-1}$ and $DT = \twobytwo{\fpd{s_2}{s_0}}{\fpd{s_2}{u_0}}{\fpd{u_2}{s_0}}{\fpd{u_2}{u_0}}$ with 
\begin{align*}
\fpd{s_2}{s_0} &= \frac{\kappa_0\ell_1\sin(2\chi_{0,2}-\theta_1) - \sin(\theta_0)\sin(2\chi_{0,2}-\theta_1)-\kappa_0\ell_2\cos(\chi_{0,2})\sin(\theta_1)}{\sin(\theta_1)\sin(\theta_2)} \\
\fpd{s_2}{u_0} &= \frac{\ell_1\sin(2\chi_{0,2}-\theta_1) - \ell_2\cos(\chi_{0,2})\sin(\theta_1)}{\sin(\theta_0)\sin(\theta_1)\sin(\theta_2)} \\
\fpd{u_2}{s_0} &= \frac{\kappa_2\sin(\theta_0)\sin(2\chi_{0,2}-\theta_1)}{\sin(\theta_1)} + 
\frac{2\sin(\chi_{0,2}) \sin(2\chi_{0,2}-\theta_1-\theta_2) (\kappa_0\ell_1-\sin(\theta_0))}{\ell_2\sin(\theta_1)} \\
&\;\;\;\;\; -\kappa _0 \left(\sin(2\chi_{0,2}-\theta_2)+\frac{\kappa_2 \ell_1\sin(2\chi_{0,2}-\theta_1)}{\sin(\theta_1)}  - \kappa _2\ell_2\cos(\chi_{0,2}) \right)\\
\fpd{u_2}{u_0} &= \frac{2\ell_1\sin(\chi_{0,2} )\sin(2\chi_{0,2}-\theta_1-\theta_2)-\kappa_2\ell_1\ell_2 \sin(2 \chi_{0,2} -\theta _1)}{\ell_2\sin(\theta_0)\sin(\theta_1)} \\
&\;\;\;\;\; + \frac{ \kappa _2 \ell_2 \cos (\chi_{0,2})-\sin(2\chi_{0,2}-\theta _2)}{\sin(\theta_0)}.
\end{align*}
Furthermore, $\det(DT) =1$. 
\label{JacobianProp} 
\end{prop}

\begin{remark}
The preceding proposition shows that the entries of the derivative matrix $DT$ do not directly depend upon the magnetic field or Larmor radius $\mu$. However, simple geometry yields the equation $\ell_{2i+2} = 2\mu \sin(\chi_{2i,2i+2})$, so the quantities above can be stated using $\mu$. Our analysis in the next sections makes use of this equation. 
\end{remark}

The derivative matrix $DT$ from Prop. \ref{JacobianProp}, interpreted as a linearization of $T$, can be used to determine the stability of periodic trajectories. If we let $z = (s,u)$, the 
\emph{stability matrix} is 
\begin{equation}
S_n(z) = DT(T^{n-1}(z)) \cdots DT(T(z)) DT(z).
\end{equation}
If $z$ is part of an orbit of period $n$, the orbit is \emph{hyperbolic} if $|\Tr S_n(z)| >2$, \emph{parabolic} if $|\Tr S_n(z)| =2$, and \emph{elliptic} if $|\Tr S_n(z)| <2$. To simplify notation, we will drop the dependence upon $z$ and largely refer to $\Tr S_n$.

\section{Linear Stability of 2-Periodic Trajectories}
\label{LinStab2}

The linear stability of standard billiards has been studied recently \cite{KozlovTr, K2000}, though the formulation was known earlier in the 20th century \cite{BB,W1986}. In particular, the linear stability of 2-periodic billiard trajectories can be simply stated in terms of the length of the trajectory and the radii of curvature of the boundary at the two points of impact, $\rho_1$ and $\rho_2$. 

\begin{prop}[\cite{KozlovTr}]
Suppose $\rho_1 <\rho_2$ are the radii of curvature of the billiard table at the impact points of a 2-periodic billiard trajectory and let $l$ denote the distance between the two impact points. If either of the inequalities $\rho_1 < l < \rho_2$ or $l > \rho_1 + \rho_2$ are satisfied, the 2-periodic billiard trajectory will be hyperbolic, while in the case $0 < l < \rho_1$ or $\rho_2 < l  < \rho_1 + \rho_2$ the trajectory is elliptic. If $l = \rho_1$, $\rho_2$, or $\rho_1 + \rho_2$, then the trajectory is parabolic. 
\label{BilliardLinStab}
\end{prop}

The proof of this proposition follows from a generating function argument \cite{KozlovTr} or by a calculation and analysis of $\Tr S_2$. The argument does not depend upon the convexity of $\Omega$ and still holds if one or both of $\rho_1$, $\rho_2$ are negative. With this stability criteria for 2-periodic billiard trajectories in mind, we investigate the linear stability of 2-periodic trajectories in the setting of  inverse-magnetic billiards.

\subsection{2-Periodic Trajectories in a Strictly Convex Set} 
Suppose $\Omega$ is strictly convex and consider a 2-periodic inverse magnetic billiard trajectory. In particular, the periodicity means $\kappa_4 = \kappa_0$ and $\theta_4 = \theta_0$. 
Geometrically, 2-periodic trajectories are in the shape of stadia -- a rectangle capped by two semicircles on opposite sides -- and was stated as an offhand remark in \cite{G2021} (and is clear if one draws a picture). The following proposition adds more structure to this statement. 

\begin{prop}
An inverse magnetic billiard trajectory in a strictly convex set is 2-periodic if and only if the angles $\chi_{0,2}$, $\chi_{2,4}$ in successive iterations of $T$ are both equal to $\pi/2$. 
\label{2PeriodicConvexProp}
\end{prop}

\begin{proof}
Suppose a trajectory is 2-periodic and $\chi_{0,2}, \chi_{2,4} \neq \pi/2$. Without loss of generality, suppose $\chi_{0,2} > \pi/2$. Then the tangent lines to the Larmor arc $\gamma_{0,2}$ at $P_1$, $P_2$ must intersect at some point $Q$. Then the second Larmor arc $\gamma_{2,4}$ must also be tangent to these two lines at the points $P_3$ and $P_4 = P_0$ and lie within the compact set bounded by $Q$, the tangent lines, and $\gamma_{0,2}$. However, since the two Larmor arcs $\gamma_{0,2}$, $\gamma_{2,4}$ have equal radii, the only way this can happen is if $\gamma_{0,2} \cup \gamma_{2,4}$ is the complete Larmor circle of radius $\mu$, making the trajectory not 2-periodic. 

Suppose $\chi_{0,2} = \chi_{2,4} = \pi/2$. Then $\ell_1$ and $\ell_3$ are parallel and the Larmor arcs $\gamma_{0,2}$, $\gamma_{2,4}$ are semicircles. If $\ell_1 > \ell_3$, then $\partial\Omega$ cannot be convex, as $P_0$, $P_1$, and $P_4$ are collinear. A similar argument holds for $\ell_1 < \ell_3$. Thus $\ell_1 = \ell_3$ and the trajectory is a stadium.   
\end{proof}

\begin{remark}
As 2-periodic trajectories are stadia, we also mention the strict convexity of $\Omega$ necessarily implies that each $\theta_{2i}, \theta_{2i+1}$ must have angle measure strictly less than $\pi/2$. Further, it is also the case that $\ell_{2} = \ell_4 = 2\mu$. 
\label{ThetaValues}
\end{remark}

With the preceding proposition, we can simplify $DT$ from Prop. \ref{JacobianProp} into 

\begin{equation}
DT  = \left(\begin{array}{cc}
\frac{\kappa_0 \ell_1 - \sin(\theta_0)}{\sin(\theta_2)} & \frac{\ell_1}{\sin(\theta_0)\sin(\theta_2)}  \\
\frac{(\kappa_0 \ell_1 - \sin(\theta_0))(\sin(\theta_1 + \theta_2)-\kappa_2\mu\sin(\theta_1))}{\mu \sin(\theta_1)} -\kappa_0 \sin(\theta_2) &  \frac{\ell_1( \sin(\theta_1+\theta_2) - \kappa_2\mu \sin(\theta_1))}{\mu\sin(\theta_0)\sin(\theta_1)} - \frac{\sin(\theta_2)}{\sin(\theta_0)}  
\end{array} \right)
\label{DTPeriod2}
\end{equation}
which in turn allows us to compute the stability matrix $S_2(s_0, u_0)$ and its trace
\begin{equation}
\Tr S_2 = 2-2\alpha (\beta + \delta) + \alpha^2 \beta\delta \\
\label{TrEqn}
\end{equation}
with 
\begin{equation}
\alpha = \ell_1/\mu,\;\;\;\beta = \cot(\theta_0) + \cot(\theta_3),\;\;\;\delta = \cot(\theta_1) + \cot(\theta_2).
\label{alphabetadelta}
\end{equation} 
By Remark \ref{ThetaValues}, $\beta$ and $\delta$ are positive, and it follows that $\alpha$ is positive as it is the ratio of two positive quantities. A direct calculation proves the following theorem. 

\begin{theorem}
Let $\alpha$, $\beta$, $\delta$ be defined as above and suppose $\Omega$ is strictly convex. Let $m = \min\{2/\beta, 2/\delta\}$ and $M = \max\{2/\beta, 2/\delta\}$. If $\beta \neq \delta$, then the 2-periodic trajectory is
\begin{itemize}
\item parabolic if and only if $\alpha = m$, $M$ or $m+M$; 
\item elliptic if and only if $\alpha \in (0,m) \cup (M, m+M)$; 
\item hyperbolic if and only if $\alpha \in (m,M) \cup (m+M,\infty)$. 
\end{itemize}
If $\beta = \delta$, then $m=M$ and the 2-periodic trajectory is 
\begin{itemize}
\item parabolic if and only if $\alpha = m$ or $2m$; 
\item elliptic if and only if $\alpha \in (0,m) \cup (m,2m)$; 
\item hyperbolic if and only if $\alpha \in (2m,\infty)$. 
\end{itemize} 
Figure \ref{Period2StabilityDiagram} presents a graphical representation of the above statement. 
\label{Period2Stability}
\end{theorem}

As stated in Proposition \ref{BilliardLinStab}, the linear stability of a 2-periodic billiard trajectory is dependent upon the relative sizes of the length of the chord and the curvature of the boundary at the points of impact. 
In magnetic billiards, the linear stability is dependent upon the curvature at the points of impact and the associated diametric chord length (\cite{KP05, RB1985}). The linear stability criteria is precisely the same criteria as standard billiards, implying the linear stability of magnetic billiards is independent of the magnetic field. However, the linear stability of periodic trajectories in the inverse magnetic billiard setting is indirectly dependent upon the magnetic field through the term $\alpha$. And in contrast to the standard billiard or magnetic billiard settings, the linear stability of inverse magnetic billiards is not dependent upon the curvature of the boundary at the exit and re-entry points. 

\begin{figure}[th]
\begin{tikzpicture}[>=stealth']
\draw[->] (0,0) -- (10,0);
\foreach \x in {0,2.5,5,7.5}
\draw[shift={(\x,0)},color=black] (0pt,3pt) -- (0pt,-3pt);
\node[below] at (0,0) {0};
\node[above] at (1.25,0) {\footnotesize $E$};
\node[below] at (2.5,0) {$m$};
\node[above] at (2.5,0.1) {\footnotesize $P$};
\node[above] at (3.75,0) {\footnotesize $H$};
\node[below] at (5,0) {$M$};
\node[above] at (5.0,0.1) {\footnotesize $P$};
\node[above] at (6.25,0) {\footnotesize $E$};
\node[below] at (7.5,0) {$m+M$};
\node[above] at (7.5,0.1) {\footnotesize $P$};
\node[above] at (8.75,0) {\footnotesize $H$};
\node[right] at (10,0) {$\alpha$};

\draw[->] (0,-2) -- (10,-2);
\foreach \x in {0,3.33,6.67}
\draw[shift={(\x,-2)},color=black] (0pt,3pt) -- (0pt,-3pt);
\node[below] at (0,-2) {0};
\node[above] at (1.67,-2) {\footnotesize $E$};
\node[below] at (3.33,-2) {$m$};
\node[above] at (3.33,-1.9) {\footnotesize $P$};
\node[above] at (5,-2) {\footnotesize $E$};
\node[below] at (6.67,-2) {$2m$};
\node[above] at (6.67,-1.9) {\footnotesize $P$};
\node[above] at (8.33,-2) {\footnotesize $H$};
\node[right] at (10,-2) {$\alpha$};

\end{tikzpicture}
\caption{A stability diagram illustrating Theorem \ref{Period2Stability} when $\beta \neq \delta$ (above) and $\beta = \delta$ (below). }
\label{Period2StabilityDiagram}
\end{figure}
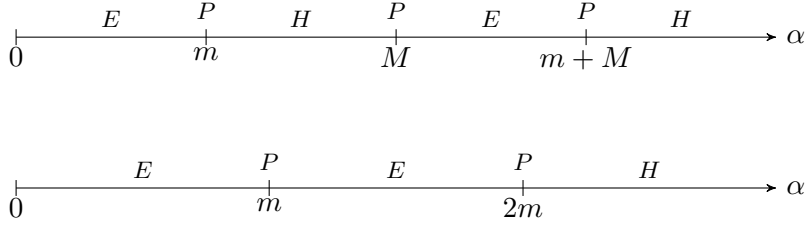

\begin{example}
Consider the case where $\partial\Omega$ is the circle of radius $R$. Then $\theta_0= \theta_1 = \theta_2 = \theta_3$ and hence $\beta = \delta$. It is then necessarily true that $\mu < R$, as $\mu \geq R$ makes the stadia larger than $\Omega$ itself. A quick calculation shows that $\alpha = 4/\beta = 2m$ and hence all 2-periodic trajectories in the circle are parabolic.  
\end{example}

\begin{example}\label{IMBEllipse}
Consider the case where $\partial \Omega$ is the ellipse $\mathcal{E}: x^2/a^2 + y^2/b^2 =1$ with $a>b>0$. In standard billiards, the only 2-periodic trajectories in the ellipse are the major and minor axes whose stability are hyperbolic and elliptic, respectively. In inverse magnetic billiards, the 2-periodic trajectories are aligned along and symmetric about the major and minor axes of the ellipse; that is, the centers of the Larmor circles lie on the coordinate axes of the ellipse. See figure \ref{2PeriodicEllipseFig}. By symmetry, the angles $\theta_0 = \theta_1 = \theta_3 = \theta_4$ and we are again in the case $\beta = \delta$. 

Consider the 2-periodic trajectory aligned along the major axis of the ellipse. This constrains the Larmor radius to $0 < \mu < b$. Two direct calculations show that 
\begin{equation}
\alpha = \frac{2a\sqrt{b^2-\mu^2}}{b\mu} \qquad \text{and} \qquad \cos(\theta_0) = \frac{a\mu}{\sqrt{b^4 + \mu^2(a^2-b^2)}}.
\end{equation}
As $\beta = 2\cot(\theta_0)$, this gives
\begin{equation}
\beta = \frac{2a\mu}{b\sqrt{b^2-\mu^2}}. 
\end{equation}
Combining these two equations gives us $\alpha$ in terms of $\beta$,
\begin{equation}
\alpha = \frac{a^2}{b^2}\frac{4}{\beta}.
\end{equation}
In terms of Theorem \ref{Period2Stability} and figure \ref{Period2StabilityDiagram}, this 2-periodic trajectory has $\alpha = \frac{a^2}{b^2}\cdot 2m$, and hence the trajectory is hyperbolic. 

In the case of the 2-periodic trajectory aligned along the minor axis of the ellipse, we now have $0 < \mu < a$. Repeating the above calculations yields 
\begin{equation}
\alpha = \frac{b^2}{a^2}\frac{4}{\beta}.
\end{equation}
In terms of Theorem \ref{Period2Stability} and figure \ref{Period2StabilityDiagram}, this 2-periodic trajectory has $\alpha = \frac{b^2}{a^2}\cdot 2m$, and hence the trajectory is elliptic except for when $2b^2=a^2$, in which case the trajectory is parabolic. It is noteworthy that this same obstruction to ellipticity occurs for standard billiards in the ellipse \cite{K2000}. We discuss this further in section \ref{EllipseDigression}.

\begin{figure}[thp]
\includegraphics[width=0.75\textwidth]{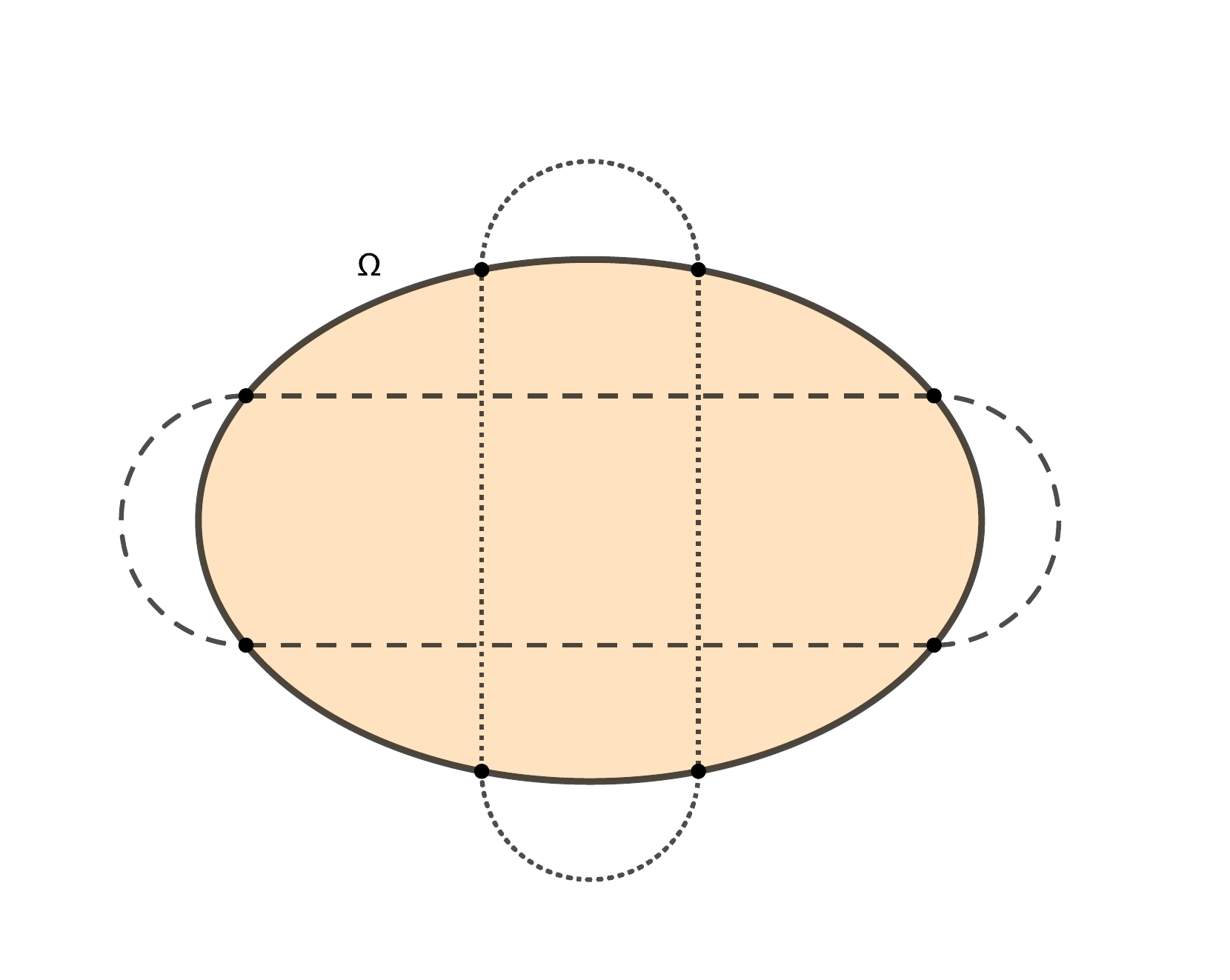}
\caption{2-periodic trajectories in the ellipse along the major axis (dashed) and minor axis (dotted). }
\label{2PeriodicEllipseFig}
\end{figure}

An interesting feature of the two families of 2-periodic trajectories is that the stability is only dependent upon the lengths of the semi-major and semi-minor axes of the ellipse. Moreover, the imposed restrictions on $\mu$ in both cases apply to the regime $\mu < \rho_{min}$ and to a subinterval of the $\rho_{min}< \mu < \rho_{max}$ regime. This adds to the limited knowledge of inverse magnetic billiards in the intermediate curvature regime. 

For comparison, consider standard billiards in the ellipse $\Omega$. The only 2-periodic billiard trajectories are the major and minor axes. A quick calculation shows that the major axis of the ellipse is hyperbolic while the minor axis is elliptic except when $a^2 = 2b^2$, in which case the trajectory is parabolic.
\label{2PeriodicEllipseExample}
\end{example}

\begin{example}
Consider the case when $\partial \Omega$ is the one-parameter family of curves $x^{2k}+y^{2k}=1$, $k \in \mathbb{N}$. As the case $k=1$ corresponds to the unit circle, we also assume $k>1$. By symmetry, we again have $\theta_0 = \theta_1 = \theta_2 = \theta_3$ and are again in the case $\beta = \delta$. 

The axes of symmetry of $\partial\Omega$ gives rise to two families of 2-periodic trajectories: the stadia are oriented along the horizontal or vertical axes of symmetry, and the stadia which are oriented along the diagonal axes of symmetry, $y=\pm x$. 

Consider first the trajectories that are oriented about the horizontal (and equivalently, vertical) axis. See figure \ref{Squarish}a. Geometrically, 2-periodic trajectories have Larmor radius $\mu$ limited to $0 < \mu < 1$. Two calculations show that
\begin{equation}
\alpha = \frac{2(1-\mu^{2k})^{\frac{1}{2k}}}{\mu} \qquad \text{and} \qquad \cos(\theta_0) = \left( (\mu^{-2k} -1)^{\frac{2k-1}{k}} + 1 \right)^{-1/2}. 
\end{equation}
Again, $\beta = 2\cot(\theta_0)$, and hence 
\begin{equation}
\beta = 2(\mu^{-2k}-1)^\frac{1-2k}{2k}.
\end{equation}
Combining these to get $\alpha$ in terms of $\beta$, we get
\begin{equation}
\alpha = \beta (\mu^{-2k}-1). 
\end{equation}
In terms of Theorem \ref{Period2Stability} and figure \ref{Period2StabilityDiagram}, we now must find the relative size of $\alpha$ with $m = 2/\beta$ and $2m$. Through some algebra, we see that $\alpha = m$ and $\alpha = 2m$ at the values
\begin{equation}
\mu^* = \left(2^\frac{k}{k-1}+1\right)^{-\frac{1}{2k}}  \qquad \text{and} \qquad \mu^{**} = 2^{-\frac{1}{2k}},
\end{equation}
respectively, and the stability can be summarized in the following way: 
\begin{itemize}
\item Parabolic if and only if $\mu = \mu^*$ or $\mu^{**}$;
\item Elliptic if and only if $\mu \in (0,\mu^*) \cup (\mu^*,\mu^{**})$;
\item Hyperbolic if and only if $\mu \in (\mu^{**},1)$.
\end{itemize}

\begin{figure}[thp]
\begin{tabular}{ c c } 
(a) \includegraphics[width=0.50\textwidth]{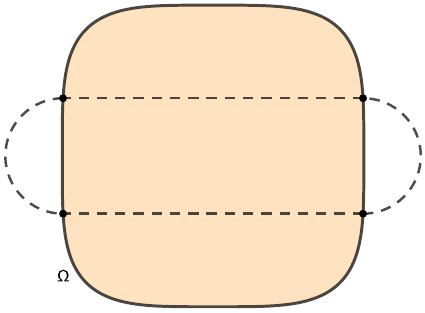} & (b) \includegraphics[width =0.40\textwidth]{2periodicSquarishDiagA.pdf} \\
\end{tabular}
\caption{2-periodic trajectories in the curve $x^{2k}+y^{2k}=1$ along the horizontal axis (a) and diagonal axis (b). }
\label{Squarish}
\end{figure}

Next, consider the case where the trajectory is oriented along the diagonal $y=x$, as in figure \ref{Squarish}b. The diagonal of $\partial \Omega$ is of length $2\cdot 2^\frac{k-1}{2k}$, so any Larmor circle must have diameter smaller than this diameter, meaning $0 < \mu < 2^\frac{k-1}{2k}$. 

We take a slightly different approach than previously. Suppose $P_0 = (x_0,y_0) = (x_0,(1-x^{2k})^{\frac{1}{2k}})$ is the point in figure \ref{Squarish}b with the largest value of $y_0$ (i.e. the point that is ``on top"). This means $x_0 \in (-2^{-\frac{1}{2k}},2^{-\frac{1}{2k}})$ and $y_0 > x_0$. In this construction,
\begin{equation}
\ell_1 = (x_0+y_0)\sqrt{2}, \qquad \mu = \frac{y_0-x_0}{\sqrt{2}}\qquad \implies \qquad \alpha = 2\frac{y_0+x_0}{y_0-x_0}.
\end{equation}
Using implicit differentiation, we find the angle $\theta_0$ can be written simply in terms of $x_0$, and $y_0$:
\begin{equation}
\tan(\theta_0) = \frac{y_0^{2k-1} + x_0^{2k-1}}{y_0^{2k-1} - x_0^{2k-1}}. 
\end{equation}
Just as in the previous examples, $\beta = 2\cot(\theta_0)$. Because $m = 2/\beta = \tan(\theta_0)$, the stability depends upon the relative values of $\alpha$ and $m =\tan(\theta_0)$ and $2m = 2\tan(\theta_0)$. Through factoring the above equation, we get
\begin{equation}
\alpha = 2m f_k(x_0,y_0)
\end{equation}
with
\begin{equation}
f_k(x_0,y_0) = \frac{y_0^{2k-2} + y_0^{2k-3}x_0+ \cdots + y_0x_0^{2k-3}+x_0^{2k-2}}{y_0^{2k-2} - y_0^{2k-3}x_0+ \cdots -y_0x_0^{2k-3}+x_0^{2k-2}}.  
\end{equation}
Thus determining the stability of the diagonal 2-periodic orbit reduces to determining the values of $f_k(x_0,y_0)$. Luckily $f_k(x_0,y_0)$ is a continuous, monotone increasing function on $(-2^{-\frac{1}{2k}},2^{-\frac{1}{2k}})$, has limiting values $1/(2k-1)$ and $2k-1$ as $x_0 \to -2^{-\frac{1}{2k}}$ and $x_0 \to 2^{-\frac{1}{2k}}$, respectively, and satisfies $f_k(0,1) = 1$. By the intermediate value theorem, there exists an $\tilde{x_0} \in (-2^{-\frac{1}{2k}},0)$ with corresponding $\tilde{y_0}$ such that $f_k(\tilde{x_0},\tilde{y_0}) = 1/2$, which in turn means $\alpha = m$. Moreover, $f_k(0,1)=1$ corresponds to $\mu = 2^{-1/2}$ and $f_k(\tilde{x_0},\tilde{y_0})=1/2$ corresponds to a value $\tilde{\mu} = (\tilde{y_0}-\tilde{x_0})/\sqrt{2}$ given by the equations above. 

Ultimately,  the stability can be summarized as follows: 
\begin{itemize}
\item Parabolic if and only if $\mu = 2^{-1/2}$ or $\tilde{\mu}$ if and only if $x_0 = 0$ or $\tilde{x_0}$;
\item Elliptic if and only if $\mu \in (0,2^{-1/2})$ if and only if $x_0 \in (0,2^{-\frac{1}{2k}})$;
\item Hyperbolic if and only if $\mu \in (2^{-1/2}, 2^\frac{k-1}{2k})$ if and only if $x_0 \in (-2^{-\frac{1}{2k}},0)\setminus\{\tilde{x_0}\}$.
\end{itemize} 

For comparison, we can apply Proposition \ref{BilliardLinStab} to this family of curves to study the linear stability of 2-periodic billiard trajectories. Such 2-periodic billiard trajectories are the horizontal and vertical axes of symmetry, and the diagonals, which are the widths and diameters of $\Omega$, respectively. 

Consider first a width of $\Omega$ as the horizontal segment connecting the points $(-1,0)$ and $(0,1)$. As the curvature of $\partial\Omega$ vanishes at these points, the radii of curvature are infinite. An analysis of the billiard stability matrix yields
\begin{equation} 
\Tr S_2 = 2 - 4l \left( \frac{1}{\rho_1} + \frac{1}{\rho_2} \right) + \frac{4l^2}{\rho_1\rho_2},
\end{equation}
and hence this trajectory is parabolic (that is, Proposition \ref{BilliardLinStab} does not directly apply at points of infinite radii of curvature). 

Consider next a diameter of $\Omega$ as the segment connecting the points $(-2^{-\frac{1}{2k}},-2^{-\frac{1}{2k}})$ and $(2^{-\frac{1}{2k}},2^{-\frac{1}{2k}})$. At these points the radii of curvature of $\partial\Omega$ are at their minimum, $\rho_1 = \rho_2 = 2^{\frac{1-k}{2k}}/(2k-1).$ Because $l = 2\cdot 2^\frac{k-1}{2k} = 2^{\frac{3k-1}{2k}}$, it is quick to show that $l > 2\rho_1$ for all $k>1$ and hence the diagonal 2-periodic billiard is hyperbolic. 
\end{example}

\subsection{2-Periodic Trajectories in a Non-Strictly Convex or Concave Set}
We again consider the case when $n=2$ and we suppose $\Omega$ is either not strictly convex (i.e. contains points of vanishing curvature) or concave (i.e. contains points of negative curvature). We address the same propositions as the convex case. 

\begin{prop}
Let $\Omega$ be convex or concave. If an inverse magnetic billiard trajectory is 2-periodic, then the angles $\chi_{0,2} = \chi_{2,4} = \pi/2$. 
\label{2PeriodicConcaveProp}
\end{prop}

The proof of this is identical to the first part of the proof of Proposition \ref{2PeriodicConvexProp}. With the assumption that $\Omega$ can be convex or concave, this now allows $\theta_i$ to be $\geq \pi/2$. The equations for $DT$, $\Tr S_2$, $\alpha$, $\beta$, and $\delta$ are unchanged but now we have the possibility that $\beta$ and $\delta$ could be nonpositive. Again, a direct calculation proves the following. 

\begin{theorem}
Let $\alpha$, $\beta$, $\delta$ be defined as in equation \ref{alphabetadelta} and suppose $\Omega$ is not strictly convex or is concave. 
\begin{enumerate}[(i) ]
\item If $\beta = \delta = 0$, then the trajectory is parabolic for all $\alpha \in (0,\infty)$; 
\item If $\beta \leq 0$ and $\delta \leq 0$ and $\beta$ and $\delta$ are not both zero,
then the trajectory is hyperbolic for all $\alpha \in (0,\infty)$; 
\item If either
	\begin{enumerate}[a)]
	\item $\beta > 0$ and $\delta=0$; or
	\item $\beta > 0$ and $\delta < 0$, and $\displaystyle \frac{2}{\beta} +\frac{2}{\delta} \leq 0$,
	\end{enumerate}
then the trajectory is 
	\begin{itemize}
	\item parabolic if and only if $\alpha = \dfrac{2}{\beta}$;
	\item elliptic if and only if $\alpha \in \left(0,\dfrac{2}{\beta}\right)$;
	\item hyperbolic if and only if $\alpha \in \left(\dfrac{2}{\beta}, \infty\right)$; 
\end{itemize}
	
\item If $\beta >0$ and $\delta <0$ and $\dfrac{2}{\beta} + \dfrac{2}{\delta} >0$, then the trajectory is 
	\begin{itemize}
	\item parabolic if and only if $\alpha = \dfrac{2}{\beta}$ or $\alpha = \dfrac{2}{\beta} + \dfrac{2}{\delta}$;
	\item elliptic if and only if $\alpha \in \left( \dfrac{2}{\beta} + \dfrac{2}{\delta}, \dfrac{2}{\beta} \right)$;
	\item hyperbolic if and only if $\alpha \in \left(0, \dfrac{2}{\beta} + \dfrac{2}{\delta}\right) \cup \left(\dfrac{2}{\beta}, \infty \right)$; 
	\end{itemize}

\item If $\beta>0$ and $\delta>0$, then the stability is determined by the statement of Theorem \ref{Period2Stability}.
\end{enumerate}
Identical statements to (iii) and (iv) hold when $\beta$ and $\delta$ are switched. 
\label{Period2ConcaveStability}
\end{theorem}

As noted in the literature on the stability of 2-periodic billiard trajectories (e.g. \cite{KozlovTr, K2000, W1986}), the stability criteria therein allow for the cases where the boundary is convex or concave at the point of reflection. As in the strictly convex case, the stability criterion for inverse magnetic billiards is only dependent upon the angles $\theta_i$, and the ratio $\ell_1/\mu$. These criteria are again not dependent upon the curvature of the boundary at the exit or reentry points nor is it directly dependent on the magnetic field. 

\begin{example}
Suppose $\Omega$ is the stadium: a rectangle with side lengths $L$ and $2R$ capped by semicircles of radius $R$ on opposite sides (see figure \ref{2PeriodicTelephoneFig}a). If the points $P_0, P_1, P_2, P_3$ of a 2-periodic trajectory are all on the straight sides of $\partial\Omega$ with $P_0, P_3$ on one side and $P_1,P_2$ on the opposite side, then the angles $\theta_i = \pi/2$ for $i = 0,1,2,3$ and necessarily $2\mu < L$. Then $\beta =\delta =0$ and the trajectory is parabolic.   

If instead the trajectory has the points $P_0,$ $P_3$ on one semicircular cap and the points $P_1,P_2$ on the other semicircular cap of the stadium, then necessarily $\mu < R$ and $\beta = \delta >0$ as all angles $\theta_i$ are equal. We use the second part of Theorem \ref{Period2Stability} to analyze the stability. A direct calculation yields 
\begin{equation}
\ell_1 = L  + 2\sqrt{R^2-\mu^2}, \qquad \cos(\theta_i) = \frac{\mu}{R}, \qquad m = \frac{\sqrt{R^2-\mu^2}}{\mu}.
\end{equation}
It follows that $\alpha =L/\mu + 2m>2m$, and hence the trajectory is hyperbolic for all $\mu < R$. 
\end{example}

\begin{example}
Suppose $\Omega$ is the ``telephone curve" with the 2-periodic trajectory pictured in figure \ref{2PeriodicTelephoneFig}b. Suppose $P_0$ is the lower right point on $\partial\Omega$. Then $\theta_1=\theta_2 > \pi/2$ while $\theta_0=\theta_3 < \pi/2$, and so $\delta<0$ and $\beta >0$. Thus we are in either case (iii) or (iv) of Theorem \ref{Period2ConcaveStability}, depending upon the sign of $\frac{2}{\beta} + \frac{2}{\delta}$, which in turn equals $\tan(\theta_0) + \tan(\theta_1)$ in this setting. Knowing or calculating the values of $\ell_1$, $\mu$, $\theta_0,$ and $\theta_1$ will then determine the stability of the pictured trajectory. 
\end{example}

\begin{figure}[thp]
\begin{tabular}{ c c } 
(a) \includegraphics[width =0.5\textwidth]{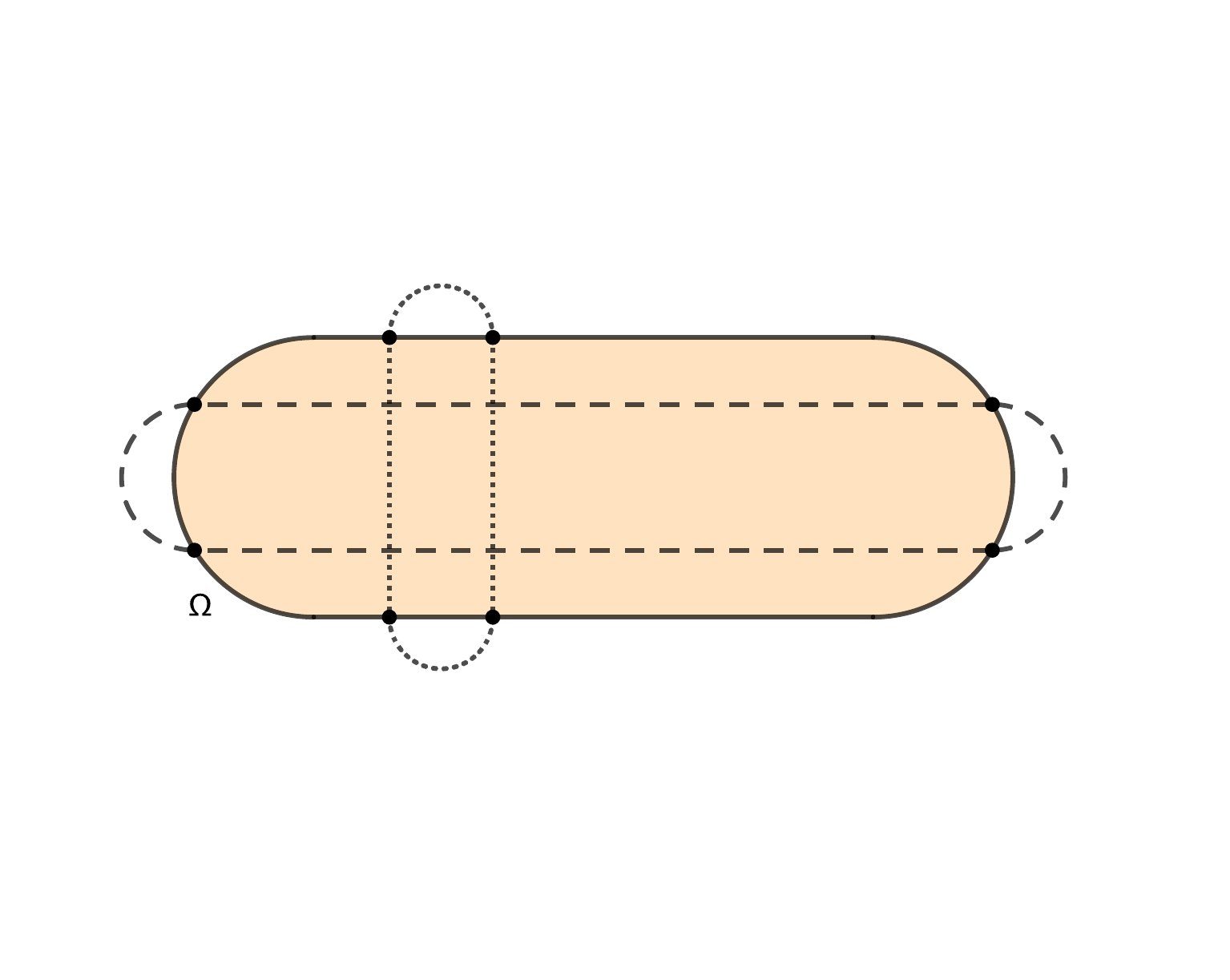} & (b) \includegraphics[width =0.40\textwidth]{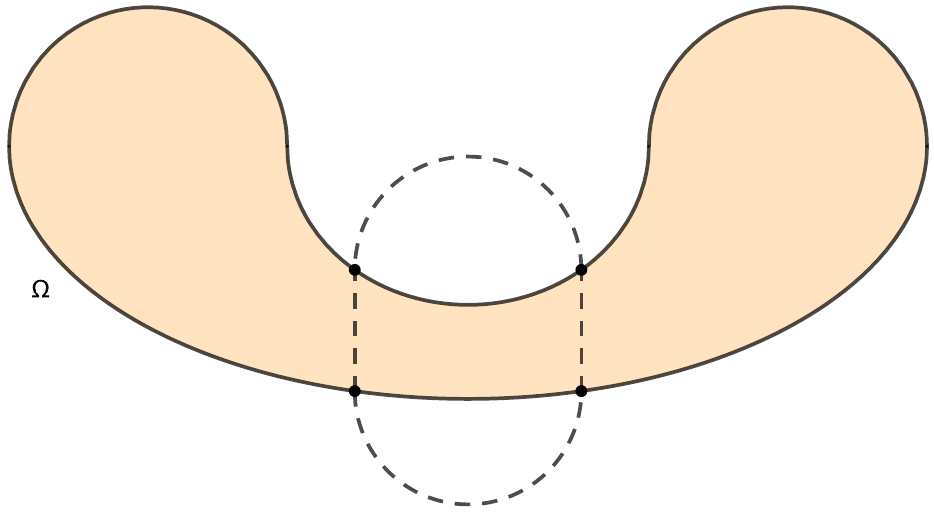} \\
\end{tabular}
\caption{(a) A parabolic (dotted) and hyperbolic (dashed) 2-periodic trajectory in the stadium ; (b) A 2-periodic trajectory in the concave ``telephone curve". }
\label{2PeriodicTelephoneFig}
\end{figure}

\subsection{A Digression on the Ellipse and Rotation Numbers}\label{EllipseDigression}

Example \ref{IMBEllipse} illustrates that the linear stability criterion for inverse magnetic billiards in an ellipse is the same as standard billiards in an ellipse: the major axis-aligned 2-periodic inverse magnetic billiard is hyperbolic and so is the standard billiard trajectory along the major axis; the minor axis-aligned 2-periodic inverse magnetic billiard trajectory is elliptic, as is the standard billiard trajectory along the minor axis, and both fail to be elliptic when $a^2 = 2b^2$. We investigate this exception further in the context of standard billiards. 

The standard billiard in an ellipse is a well-known example of an integrable system and is conjectured to be the only integrable planar billiard \cite{Bir}. Billiard trajectories in the ellipse have the following caustic property: if one segment of a trajectory is tangent to an ellipse confocal to the billiard table, then after every reflection the trajectory will remain tangent to the same confocal ellipse; if one segment (or its extension) of a trajectory is tangent to a hyperbola which is confocal to the billiard table, then every segment or its extension will be tangent to the same confocal hyperbola; and if a segment of the billiard trajectory passes through one focus, then after reflection the trajectory will pass through the other focus. The confocal ellipse or hyperbola with the aforementioned tangency property are called \emph{caustics} of the billiard trajectory. If we write the ellipse $\mathcal{E}: x^2/a^2+y^2/b^2=1$ for $a>b>0$, then the confocal family of $\mathcal{E}$ is given by 
\begin{equation}
\mathcal{E}_\lambda: \frac{x^2}{a^2-\lambda} + \frac{y^2}{b^2-\lambda} = 1
\end{equation}
for $\lambda \in \R{}$. The caustics are members of the confocal family which are ellipses for $\lambda \in (0,b^2)$ and hyperbolas for $\lambda \in (b^2,a^2)$. The caustics are degenerate and are contained in the major and minor axes for $\lambda = b^2$ and $a^2$, respectively. If $\lambda <0$ or $\lambda > a^2$, then the curves $\mathcal{E}_\lambda$ are outside $\mathcal{E}$ or imaginary, respectively. 

Given a periodic billiard trajectory in the ellipse $\mathcal{E}$, define the following \emph{rotation function } as the ratio of elliptic integrals
\begin{equation}
\rot: \left(0, b^2\right) \cup \left(b^2, a^2\right) \to \R{}, \qquad \rot(\lambda) = \frac{\displaystyle \int_0^{\min(b^2,\lambda)} \frac{dt}{\sqrt{(\lambda - t)(b^2 - \lambda)(a^2 - \lambda)}}}{\displaystyle 2\int_{\max(b^2,\lambda)}^{a^2}\frac{dt}{\sqrt{(\lambda - t)(b^2 - \lambda)(a^2 - \lambda)}}}.
\end{equation}
See e.g. \cite{CRR2012,DR2019} for a modern treatment, though these elliptic integrals were known to Jacobi and other mathematicians of the 19th century. 

A periodic trajectory will have $\rot(\lambda) = m/n \in \Q$ for coprime integers $m \geq n$. When $\mathcal{E}_\lambda$ is an ellipse, $m$ is the winding number and $n$ is the minimal period of the trajectory. 
In particular, $\rot(\lambda)$ increases monotonically from 0 to 1 as $\lambda$ increases monotonically from 0 to $b^2$. If $\mathcal{E}_\lambda$ is a hyperbola, then $m$ is the number of times the trajectory crosses the $y$-axis and $n$ is the minimal period. This necessarily implies the trajectory must have even period when the caustic is a hyperbola. As $\lambda$ increases monotonically from $b^2$ to $a^2$, $\rot(\lambda)$ decreases monotonically from 1 to $r>0$, where $r$ can be calculated explicitly, as we will see below. Proving the (decreasing) monotonicity of $\rot(\lambda)$ when the caustic is a hyperbola is nontrivial and was proved in \cite{D}. A simpler proof was given recently \cite{DR2019} using a different technique. 

To determine the limiting value of $\rot(\lambda)$ for a hyperbolic caustic, $r>0$, consider the approach of \cite{D}, where the confocal family has fixed foci at $(\pm1,0)$ and can be written in the form: 
\begin{equation}
\mathcal{C}_{\nu}: \frac{x^2}{\nu} + \frac{y^2}{\nu-1}=1. 
\end{equation}
The confocal curves are hyperbolas for $\nu \in (0,1)$ and ellipses for $\nu \in (1,\infty)$. Any ellipse of the form $\mathcal{E}: x^2/a^2 + y^2/b^2=1$ can be homothetically scaled to a member of the confocal family $\mathcal{C}_{\nu_0}$ for some $\nu_0>1$. Suppose the billiard boundary is the member of this confocal family $\mathcal{C}_\nu$ corresponding to $\nu = \nu_0>1$.  Then the limiting rotation number for a hyperbolic caustic is only dependent upon $\nu_0$ and is given by 
\begin{equation}
r = \displaystyle \lim_{\lambda \nearrow a^2} \rot(\lambda) = \frac{1}{\pi} \arccos\left( 1 - \frac{2}{\nu_0} \right). 
\end{equation}
See Section 11.2 and specifically subsection 11.2.3.5 of \cite{D} for details. 

In the case of the isolated parabolic billiard trajectory along the minor axis of the ellipse, the ellipse $\mathcal{E}$ has parameters $a$ and $b$ satisfying $a^2 = 2b^2$. The confocal family $\mathcal{C}_\nu$ has $a^2 = \nu_0$ and $b^2 = \nu_0-1$, so that together these three equations imply that the billiard table $\mathcal{C}_{\nu_0}$ has $\nu_0 = 2$. Moreover, this implies that $r = 1/2$. The conclusion we draw from this is the isolated parabolic billiard trajectory in the minor axis of the ellipse has rotation number $1/2$, and that as $\lambda$ increases from $b^2$ to $a^2$, the rotation function $\rot(\lambda)$ decreases monotonically from $1$ to $1/2$. That is, the isolated parabolic billiard trajectory along the minor axis of the ellipse is characterized by being the limiting trajectory of the unique elliptical billiard table whose trajectories with hyperbolic caustics have limiting rotation number $1/2$.

\section{Linear Stability of 3- and 4-Periodic Trajectories} 
\label{LinStab34}

\subsection{3-Periodic Trajectories} The study of 3-periodic trajectories can be performed in a similar fashion. However, we now have two distinct types of 3-periodic trajectories: those with rotation number $1/3$ and those with rotation number $2/3$. For an arbitrary 3-periodic trajectory, $\Tr S_3$ is cubic in $1/\mu$ with coefficients that are expressions in terms of $\ell_1$, $\ell_3$, $\ell_5$, and sines and cosines of combinations of $\chi_{0,2}$, $\chi_{2,4}$, $\chi_{4,6}$, and $\theta_i$ for $i = 0,\ldots, 5$. As in the 2-periodic case, no curvature terms $\kappa_i$ appear. 

\begin{figure}[thp]
\includegraphics[width=0.75\textwidth]{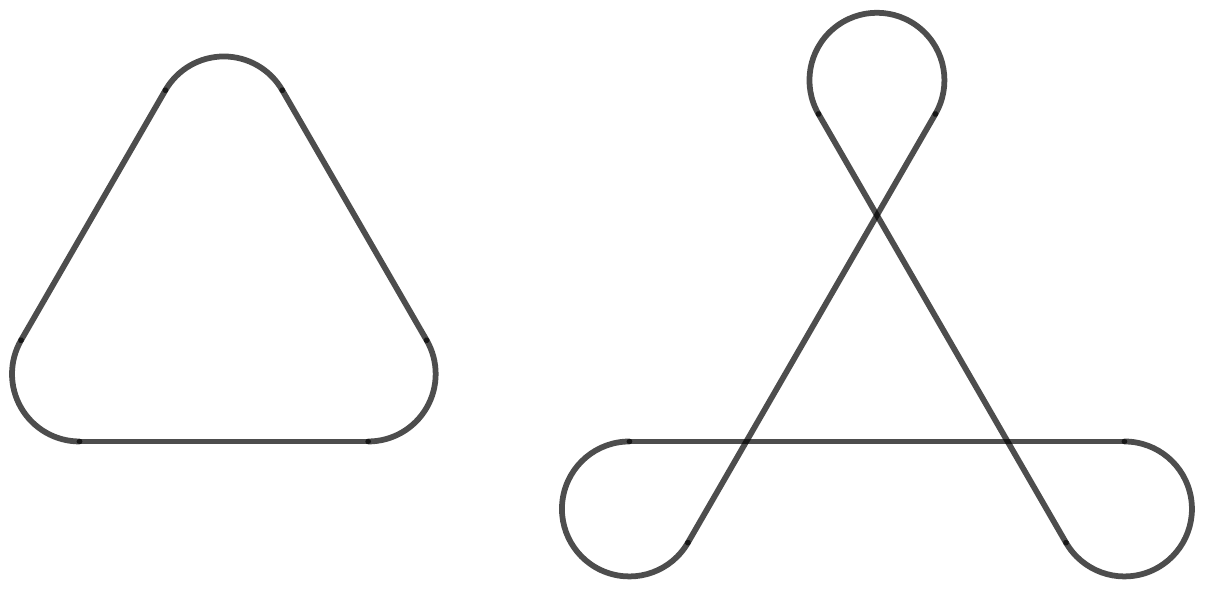}
\caption{The shape of 3-periodic trajectories with dihedral symmetry with rotation number $1/3$ (left) and $2/3$ (right). }
\label{Period3Shapes}
\end{figure}

To make the 3-periodic trajectories simpler to study, consider 3-periodic trajectories with dihedral symmetry. See figure \ref{Period3Shapes}. This assumption means $\ell_1 = \ell_3 = \ell_5 =: \ell$ and $\chi_{0,2} = \chi_{2,4} = \chi_{4,6} =: \chi$ are both constant with the latter equal to $\pi/3$ or $2\pi/3$ when the trajectory has rotation number $1/3$ or $2/3$, respectively. Then $\Tr S_3$ is cubic in $\alpha := \ell/\mu$ and has coefficients which can be written in terms of sums, products, and differences of cotangents of the $\theta_i$. For example, the constant term of $\Tr S_3$ can be written as 
\begin{align}
2 - \frac{3}{4} C_{23}C_{45} -\frac{3}{8} C_{01}(2C_{2345} \mp \sqrt{3}C_{23}C_{45})
\end{align}
where $C_{A} = \sum_{i \in A} \cot(\theta_i)$, the minus sign corresponds to the trajectory with rotation number $1/3$, and the plus sign corresponds to the trajectory with rotation number $2/3$. The cubic coefficient can be written as 
\begin{align}
-\mp \frac{\cos(\frac{\pi}{6} \mp (\theta_1 + \theta_2)) \cos(\frac{\pi}{6} \mp (\theta_3 + \theta_4)) \cos(\frac{\pi}{6} \mp (\theta_0 + \theta_5))}{\sin(\theta_0)\sin(\theta_1) \sin(\theta_2) \sin(\theta_3)\sin(\theta_4)\sin(\theta_5)}
\end{align}
with the same sign convention stated above applying to all $\mp$ terms. 

Ultimately, the linear stability of 3-periodic trajectories is difficult to analyze in general. 

\begin{example}
Consider the case when all $\theta_i=:\theta$ are equal. Then 
\begin{align*}
\Tr S_3 (s,\theta) &= 2-9\cot^2(\theta) - 3\sqrt{3}\cot^3(\theta)  \\
&\;\;\;\;\; + \alpha \left[\frac{3 \cos(\theta)}{4\sin^4(\theta)} \left(5\sqrt{3} \cos(\theta) + \sqrt{3}\cos(3\theta) - 3\sin(\theta) + 9\sin(3\theta) \right) \right] \\
&\;\;\;\;\; + \alpha^2 \left[ -\frac{3\sin(\frac{\pi}{3}+2\theta)(\cos(\theta)+\sin(\frac{\pi}{6}+3\theta))}{\sin^5(\theta)}\right] \\
&\;\;\;\;\; + \alpha^3 \left[ \frac{\cos^3(\frac{\pi}{6}-2\theta)}{\sin^6(\theta)} \right]
\end{align*}
and 
\begin{align*}
\Tr S_3 (s,\theta) &= 2-9\cot^2(\theta) + 3\sqrt{3}\cot^3(\theta) \\
&\;\;\;\;\; +\alpha \left[\frac{3 \cos(\theta)}{4\sin^4(\theta)} \left(-5\sqrt{3} \cos(\theta) - \sqrt{3}\cos(3\theta) - 3\sin(\theta) + 9\sin(3\theta) \right) \right] \\
&\;\;\;\;\; + \alpha^2 \left[ \frac{3\sin(\frac{\pi}{3}-2\theta)(\cos(\theta)+\sin(\frac{\pi}{6}-3\theta))}{\sin^5(\theta)}\right]  \\
&\;\;\;\;\; + \alpha^3 \left[ -\frac{\cos^3(\frac{\pi}{6}+2\theta)}{\sin^6(\theta)} \right] \\
\end{align*}
when the trajectory has rotation number $1/3$ and $2/3$, respectively. Once $\theta$ and $\ell$ are known, the value of $|\Tr S_3|$ can be computed to determine the stability. 
\end{example}

\begin{example}
Consider the case when $\partial\Omega$ is the circle of radius $R$. By symmetry, $\theta_i$ is a constant of motion and we are in the same situation as the previous example. First, suppose the trajectory has a rotation number $1/3$ and hence $\chi = \pi/3$. Then for each $\mu < R$, there is a unique $\theta \in (0,\pi/3)$ such that $(s,\theta)$ is the initial condition for the 3-periodic trajectory. In fact, we can say that this $\theta$ satisfies the equation 
\begin{equation}
\sin\left(\frac{\pi}{3}-\theta\right) = \frac{\mu \sin(\theta)}{\sqrt{R^2+\mu^2 - 2R\mu\cos(\theta)}}
\end{equation}
(see \cite{G2021}) which has functional solution
\begin{equation}
\cos(\theta) = \frac{3\mu + \sqrt{4R^2-3\mu^2}}{4R}.
\end{equation} 
In the case of the circle, $\ell = 2R\sin(\theta)$, which turns equation for $\Tr S_3$ into an equation in terms of $R$ and $\mu$. In particular, the equation $|\Tr S_3|=2$ is satisfied for all $\mu <R$, so the trajectory is parabolic. 

Next, consider the case when the rotation number is $2/3$ and so $\chi = 2\pi/3$. Again, for each $\mu < R$ there exists a unique $\theta \in (\pi/3,2\pi/3)$ such that $(s,\theta)$ is the initial condition for a 3-periodic trajectory. Just as before, this value of $\theta$ satisfies
\begin{equation}
\sin\left(\frac{\pi}{3}-\theta\right) = \frac{\mu \sin(\theta)}{\sqrt{R^2+\mu^2 - 2R\mu\cos(\theta)}}
\end{equation}
and has functional solution
\begin{equation}
\cos(\theta) = \frac{3\mu - \sqrt{4R^2-3\mu^2}}{4R}.
\end{equation} 
Evaluating the expression for $\Tr S_3$ for this value of $\theta$ yields $|\Tr S_3|=2$ for all values of $\mu <R$. Therefore both types of 3-periodic trajectories in the circle are parabolic. 
\end{example}

\subsection{4-Periodic Trajectories} 
The study of 4-periodic trajectories can be performed in a similar fashion. Again, we have two distinct types of 4-periodic trajectories: those with rotation number $1/4$ and those with rotation number $3/4$. For an arbitrary 4-periodic trajectory, $\Tr S_4$ is quartic in $1/\mu$ with coefficients that are expressions in terms of $\ell_i$ for $i \in \{1,3,5,7\}$ and sines and cosines of combinations of $\chi_{0,2}$, $\chi_{2,4}$, $\chi_{4,6}$, $\chi_{6,8}$, and $\theta_i$ for $i = 0,\ldots, 7$. As in the 2-periodic case, no curvature terms $\kappa_i$ appear. 

Again, we examine examples which incorporate symmetry to simplify the quartic expression for $\Tr S_4$. 

\begin{example}
Consider the case when $\partial\Omega$ is the circle of radius $R$. By symmetry, $\theta_i$ is a constant of motion and $\ell_1 = \ell_3 = \ell_5 = \ell_7 =: \ell$ and $\chi_{0,2} = \chi_{2,4}= \chi_{4,6}= \chi_{6,8}=:\chi$. First, suppose the trajectory has a rotation number $1/4$ and hence $\chi = \pi/4$. Then for each $\mu < R$, there is a unique $\theta \in (0,\pi/4)$ such that $(s,\theta)$ is the initial condition for the 4-periodic trajectory. In fact, we can say that this $\theta$ satisfies the equation 
\begin{equation}
\sin\left(\frac{\pi}{4}-\theta\right) = \frac{\mu \sin(\theta)}{\sqrt{R^2+\mu^2 - 2R\mu\cos(\theta)}}
\end{equation}
(see \cite{G2021}) which has functional solution
\begin{equation}
\cos(\theta) = \frac{R^2 + \mu^2 -\mu^2\csc^2(\frac{\pi}{4}-\theta)\sin^2(\theta)}{2R\mu}.
\end{equation} 
In the case of the circle, $\ell = 2R\sin(\theta)$, which turns equation for $\Tr S_4$ into an equation in terms of $R$ and $\mu$. Letting $\alpha = R/\mu$, the equation becomes
\begin{equation}
\begin{split}
\Tr S_4 &= \frac{2(1- 10\cos^2(\theta) + 17 \cos^4(\theta))}{\sin^4(\theta)} + \alpha \left[ \frac{32\cos(2\theta)(3\cos^2(\theta)-1)}{\sin^4(\theta)}\right] \\	
	&\;\;\;\;\; + \alpha^2 \left[ \frac{8 \cos^2(2\theta)(5+7\cos(2\theta))}{\sin^4(\theta)} \right]  - \alpha^3 \left[ \frac{64\cos^3(2\theta)\cos(\theta)}{\sin^4(\theta)}\right] \\
	&\;\;\;\;\; + \alpha^4 \left[ \frac{16 \cos^4(2\theta)}{\sin^4(\theta)} \right].
\end{split}
\end{equation}
For the specific value of $\theta$ which produces the 4-periodic trajectory with rotation number $1/4$, this expression for $\Tr S_4$ simplifies to $2$, and hence the trajectory is parabolic. 

An analogous calculation can be done to show the 4-periodic trajectory with rotation number $3/4$ is parabolic. 
\end{example}

\begin{example}
Consider the case when $\partial\Omega$ is the ellipse $x^2/a^2 + y^2/b^2 =1$ with $a>b>0$. Further, suppose the 4-periodic trajectories share the same horizontal and vertical axes of symmetry as the ellipse (see figure \ref{4PeriodicEllipse}). In particular, this symmetry assumption means that $\theta_0 = \theta_1 = \theta_4 = \theta_5$, $\theta_2 = \theta_3 = \theta_6 = \theta_7$, $\ell_1 = \ell_5$, $\ell_3 = \ell_7$, and $\chi_{0,2} = \chi_{2,4}= \chi_{4,6}= \chi_{6,8}=:\chi$, which in turn implies $\ell_2 = \ell_4 = \ell_6 = \ell_8$. Suppose the Cartesian coordinates $(x_0,-y_0)$ are the initial point $P_0$ in the fourth quadrant of a 4-periodic trajectory. 

\begin{figure}[thp]
\includegraphics[width = 0.75\textwidth]{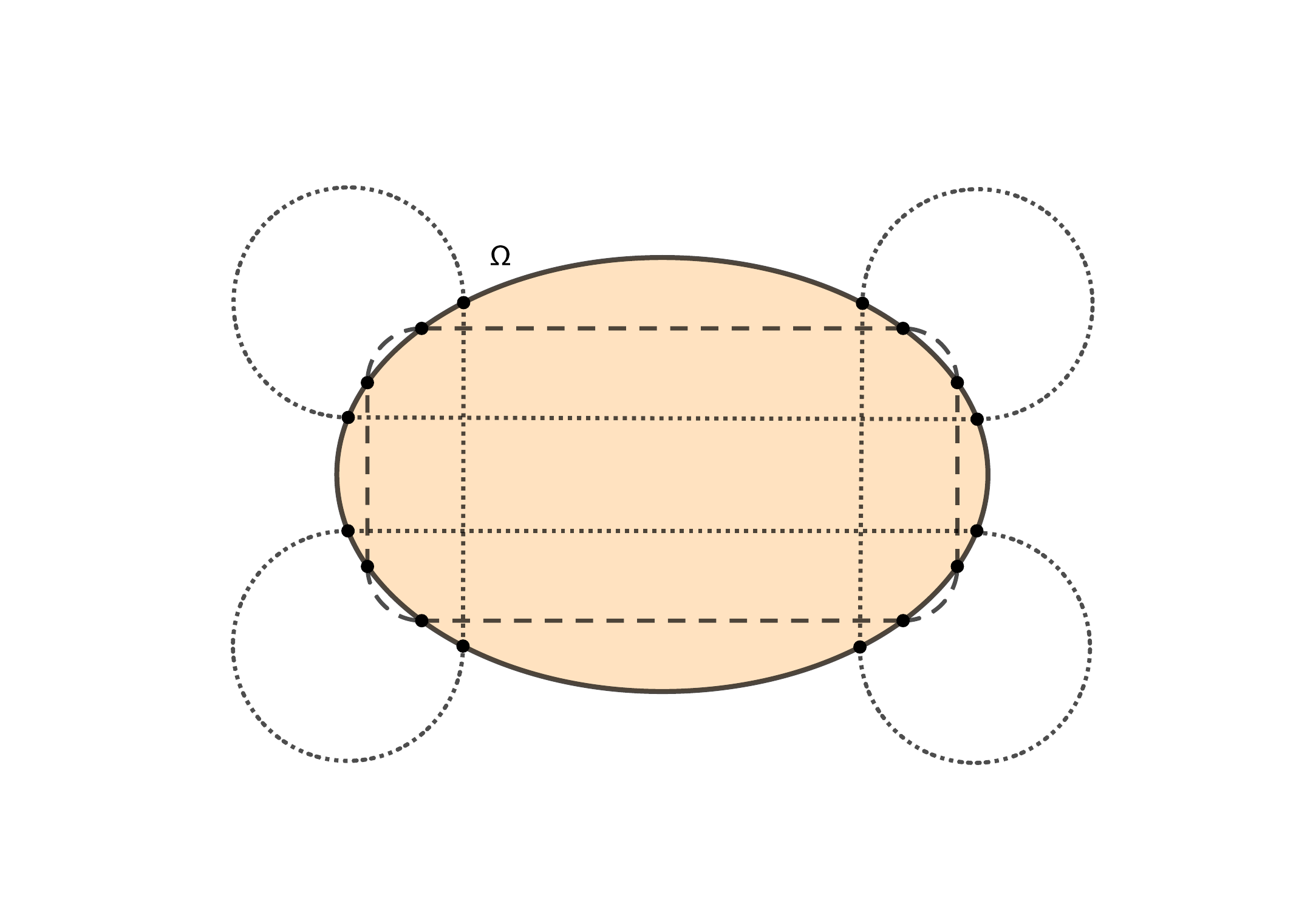}
\caption{Symmetric 4-periodic trajectories in the ellipse with rotation number $1/4$ (dashed) and $3/4$ (dotted), each with different Larmor radii.}
\label{4PeriodicEllipse}
\end{figure}

Consider first the case when the trajectory has rotation number $1/4$, and hence $\chi = \pi/4$. We aim to express the stability of the trajectory in terms of the Cartesian coordinates of $P_0$. Clearly $P_1 = (x_0,y_0)$, and the point $P_2$ will be the intersection of the ellipse with the line $y=-x +x_0 + y_0$ other than the point $P_1$, which is the same as the second intersection point of the Larmor circle. Further, this means the only possible initial $x$-coordinates are $x_0 \in (a(a^2-b^2)/(a^2+b^2),a)$ -- otherwise the quarter Larmor circle cannot intersect the ellipse to create the symmetric 4-periodic trajectory. Since $\chi=\pi/4$, we can write $P_2 = (x_2,y_2) = (x_0 - \mu, y_0 + \mu)$. A few calculations yield the following:

\begin{equation}
\begin{split}
\ell_1 &= 2y_0 \\
x_2 &= \frac{a^2(x_0+y_0) - ab \sqrt{a^2+b^2 - (x_0+y_0)^2}}{a^2+b^2} \\
y_2 &= \frac{b^2(x_0+y_0) + ab \sqrt{a^2+b^2 - (x_0+y_0)^2}}{a^2+b^2} \\
\mu &= \frac{2ab \sqrt{a^2 + b^2 - (x_0 + y_0)^2}}{a^2+b^2} = \pm \frac{2(b^2x_0 - a^2y_0)}{a^2+b^2} \\
\ell_3 &= 2x_2 \\
\cos(\theta_0) &= \frac{b^2x_0}{\sqrt{a^4y_0^2 + b^4 x_0^2}} \\
\cos(\theta_2) &= \frac{a^2y_2}{\sqrt{a^4y_2^2 + b^4 x_2^2}} = \frac{a^2(y_0+\mu)}{\sqrt{a^4(y_0+\mu)^2 + b^4 (x_0-\mu)^2}} \\
\end{split}
\label{4PeriodicEquations}
\end{equation} 
where the second equation for $\mu$ uses plus if $\chi = \pi/4$ and minus if $\chi = 3\pi/4$. 

Plugging in each of the above equations into $\Tr S_4$ yields a rational function in $x_0$ and $y_0$ with coefficients in $a$ and $b$:

\begin{equation}
\begin{split}
\Tr S_4 &= \left[  \right. 16 b^4 \left(a^2 - b^2\right)^4 x_0^4 -16 b^2 \left(a^2 - b^2\right)^3 \left(2 a^4 - 3 a^2 b^2 + 2 b^4\right) x_0^3 y_0 \\
	&\;\;\;\;\; 2 \left(a^2 - b^2\right)^2 \left(8 a^8 - 40 a^6 b^2 + 49 a^4 b^4 - 40 a^2 b^6 + 
   8 b^8\right) x_0^2 y_0^2 \\
	&\;\;\;\;\; 8 a^2 \left(a^2 - b^2\right) \left(4 a^8 - 10 a^6 b^2 + 13 a^4 b^4 - 10 a^2 b^6 + 
   4 b^8\right) x_0 y_0^3 \\
	&\;\;\;\;\; 8 a^4 \left(2 a^8 - 4 a^6 b^2 + 5 a^4 b^4 - 4 a^2 b^6 + 2 b^8\right) y_0^4 \left. \right] / \left[ a^4 b^4 y_0^2 \left(b^2 x_0-a^2 \left(x_0+2 y_0\right)\right)^2\right].
\end{split}
\label{TrS4}
\end{equation}

As $y_0$ can be written as a function of $x_0$, we can then express the stability of the 4-periodic trajectory in terms of the coordinate $x_0$. 

Repeating the above calculations in the case with rotation number $3/4$ and $\chi = 3\pi/4$ yields nearly the same equations as in (\ref{4PeriodicEquations}) and ultimately yields the same equation for $\Tr S_4$ given above in (\ref{TrS4}). See Remark \ref{4PeriodicRemark} below for additional commentary on this case. 

To further illustrate the stability analysis, consider the specific case when $a=3$ and $b=2$. Then $|\Tr S_4| = 2$ at the values 
\begin{equation}
x_0^* = \frac{291}{9\sqrt{13}}, \qquad x_0^{**} = \sqrt{\frac{88731 + 1575\sqrt{217}}{14534}}, \qquad x_0^{***} = \frac{291}{13\sqrt{61}}.
\end{equation}
Analyzing $|\Tr S_4|$ yields that the 4-periodic trajectory is 
\begin{itemize}
\item elliptic for $x_0 \in (x_0^*,x_0^{**})\cup (x_0^{**},x_0^{***})$;
\item parabolic for $x_0 \in \{x_0^*,x_0^{**},x_0^{***}\}$; and 
\item hyperbolic for $x_0 \in (15/13,x_0^*) \cup (x_0^{***},3).$
\end{itemize}

\begin{remark}
The symmetric 4-periodic trajectories in the ellipse with rotation number $1/4$ and $3/4$ are related to one another in the following way. Consider the eight ordered points $P_0, P_1, \ldots,P_7$ on $\partial\Omega$ which determine a trajectory with rotation number $1/4$. Then a trajectory in the order $P_0, P_5, P_6, P_3, P_4, P_1, P_2, P_7$ is a symmetric 4-periodic trajectory with rotation number $3/4$. See figure \ref{4PeriodicEllipseDuality}. This is partly due to the fact that for each segment $\ell_{2j}$, there are two supplementary values of $\chi$ which realize this segment. This in turn creates two Larmor arcs whose union (after a suitable reflection across $\ell_{2j}$) is a complete Larmor circle. As such, these two 4-periodic trajectories are complementary to one another. 
\label{4PeriodicRemark}
\end{remark}

\begin{figure}[thp]
\includegraphics[width = 0.75\textwidth]{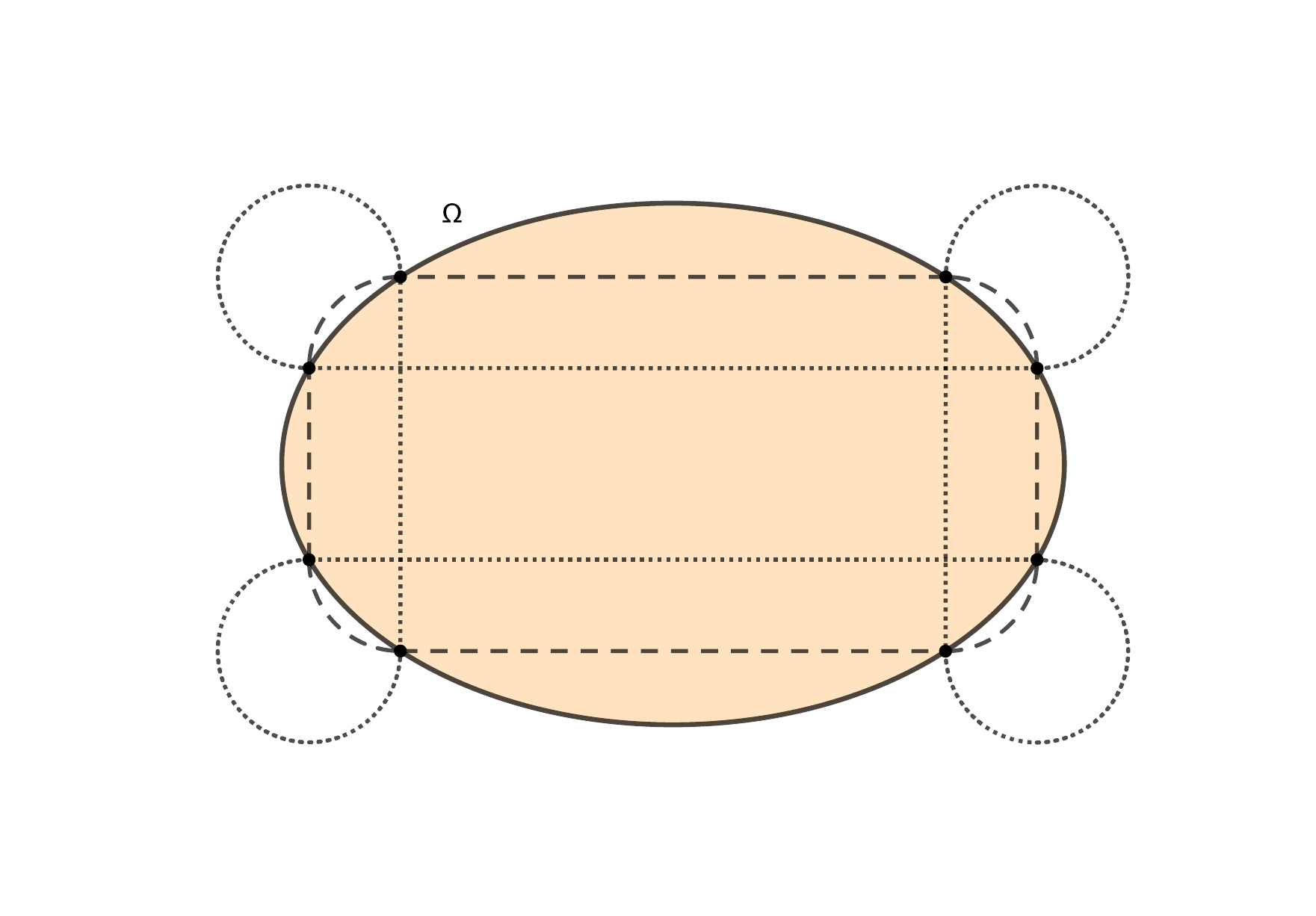}
\caption{Symmetric 4-periodic trajectories in the ellipse with rotation number $1/4$ (dashed) and $3/4$ (dotted), each with the same Larmor radii and same points $P_i$. }
\label{4PeriodicEllipseDuality}
\end{figure}

\begin{remark}
There is a second family of symmetric 4-periodic trajectories in the ellipse: those whose Larmor circles are symmetric about the axes of symmetry of the ellipse. The stability of these trajectories can be analyzed in a similar fashion to those above, though we do not address this example here.  
\end{remark}
\end{example}

\begin{example}
Consider the case when $\partial \Omega$ is the one-parameter family of curves $x^{2k}+y^{2k}=1$, $k \in \mathbb{N}$. As the case $k=1$ corresponds to the unit circle, we also assume $k>1$. By the dihedral symmetry of $\partial \Omega$, the 4-periodic trajectories have constant $\theta_i :=\theta$, $\ell_{2j}$, $\ell_{2j+1}$, and $\chi_{2l,2l+2} =:\chi$. In addition to the two families of trajectories determined by the rotation number, we can also consider two further families based upon whether the centers of the Larmor circles are on the diagonal axes of symmetry of $\partial \Omega$ or on the horizontal and vertical axes of symmetry of $\partial\Omega$.

First, consider the case when the rotation number is $1/4$ and the centers of the Larmor circles are on the diagonal axes of symmetry of $\partial\Omega$. Suppose the point $P_0 = (x_0,-y_0) \in \partial\Omega$ is in the fourth quadrant and the trajectory has initial velocity in the direction $(0,1)$ so that $P_1 = (x_0,y_0) \in \partial\Omega$ is in the first quadrant. To produce the 4-periodic trajectory, it must be the case that $x_0 \in (2^{-\frac{1}{2k}},1)$ and $P_2 = (y_0,x_0)$. Further, we know $\chi = \pi/4$, which implies the center of the first Larmor arc is at $(y_0,y_0)$ and hence $\mu = x_0 - y_0$. 

Due to the curvature of $\partial\Omega$, it is possible for the Larmor arc of this first segment of the trajectory to intersect $\partial\Omega$ in up to four different locations. To ensure the point $P_2$ is the proper point of intersection of the Larmor circle with $\partial\Omega$, there must be exactly two intersection points between the Larmor circle and $\partial\Omega$ so that the billiard can ``round the corner" of $\partial\Omega$. That is, there exists $\hat{x} \in (2^{-\frac{1}{2k}},1)$ such that 
\begin{equation}
(x-y_0)^2+(y-y_0)^2 = \mu^2 \qquad \text{and} \qquad x^{2k}+y^{2k}=1
\label{IntersectionEquations}
\end{equation}
has exactly three real roots in $x$: $x=\hat{x}$, $x=2^{-\frac{1}{2k}}$, and $x=y_0$, with $y_0 < 2^{-\frac{1}{2k}}< \hat{x}$. Alternately, we can say that $\hat{x}$ is the value of $x_0$ such that the distance from the center of the Larmor circle, $(y_0,y_0)$ and $(2^{-\frac{1}{2k}},2^{-\frac{1}{2k}})$ is exactly $\mu$. For this value $\hat{x}$, these two equations (\ref{IntersectionEquations}) have exactly two real roots for $x_0 \in (2^{-\frac{1}{2k}},\hat{x})$, three real roots when $x_0 = \hat{x}$, and four real roots when $x_0 \in (\hat{x},1)$. This behavior is analogous to a circle whose center lies on the diagonal of a square and can have multiple intersections with the sides and corners of the square. We now assume $x_0 \in (2^{-\frac{1}{2k}},\hat{x})$. 

We can now compute several of the relevant quantities that appear in $\Tr S_4$ :
\begin{equation}
\begin{split}
\ell_1 &= 2y_0 \\
\mu &= x_0-y_0 \\
\ell_2 &= 2\mu \\
\tan(\theta) &= \frac{y_0^{2k-1}}{x_0^{2k-1}}. \\
\end{split}
\label{4PeriodicEquationsAgain}
\end{equation} 
These quantities produce the following expression for $\Tr S_4$ in terms of $x_0$ and $y_0$:
\begin{align}
\begin{split}
\Tr S_4 &= 2 + \left[ 16x_0^2 \left(x_0^{2k-2} -y_0^{2k-2}\right)\left(x_0^{4k-2} - y_0^{4k-2}\right) \left(x_0^{2k-2} - y_0^{2k-2} + 2x_0^{-2}y_0^{2k-1}\right)  \right. \\
	&\;\;\;\;\; \cdot \left. \left(y_0^{4k-2}-x_0^{4k-2} + (x_0-y_0)x_0^{2k-1}y_0^{2k-2}\right)^2  \right] / \left[ y_0^{16k-12}(x_0-y_0)^4\right]
\end{split}
\label{4PerTrace}
\end{align}
By assumption, $0 < y_0 < x_0 <1$, and each term in the above rational function will be strictly positive for $x_0 \in (2^{-\frac{1}{2k}},\hat{x})$. Therefore this 4-periodic trajectory will be hyperbolic for all $x_0 \in (2^{-\frac{1}{2k}},\hat{x})$ and integer $k\geq 2$. 

Repeating the above calculations in the case when the rotation number is $3/4$ and $x_0 \in (-1,2^{-\frac{1}{2k}})$ yields the same equation for $\Tr S_4$ as (\ref{4PerTrace}). Similar to before, this expression will be larger than 2 for $x_0 \in (-1,2^{-\frac{1}{2k}})\setminus \{-2^{-\frac{1}{2k}}\}$ and be equal to 2 for $x_0 = -2^{-\frac{1}{2k}}$. Therefore the 4-periodic trajectory with rotation number $3/4$ is hyperbolic for all $x_0 \in (-1,2^{-\frac{1}{2k}})\setminus \{-2^{-\frac{1}{2k}}\}$ and parabolic for $x_0 = -2^{-\frac{1}{2k}}$. There is also the same duality between this trajectory and the trajectory with rotation number $1/4$ as is discussed in Remark \ref{4PeriodicRemark}.

\begin{figure}[thp]
\begin{tabular}{ l l } 
(a) \includegraphics[width = 0.40\textwidth]{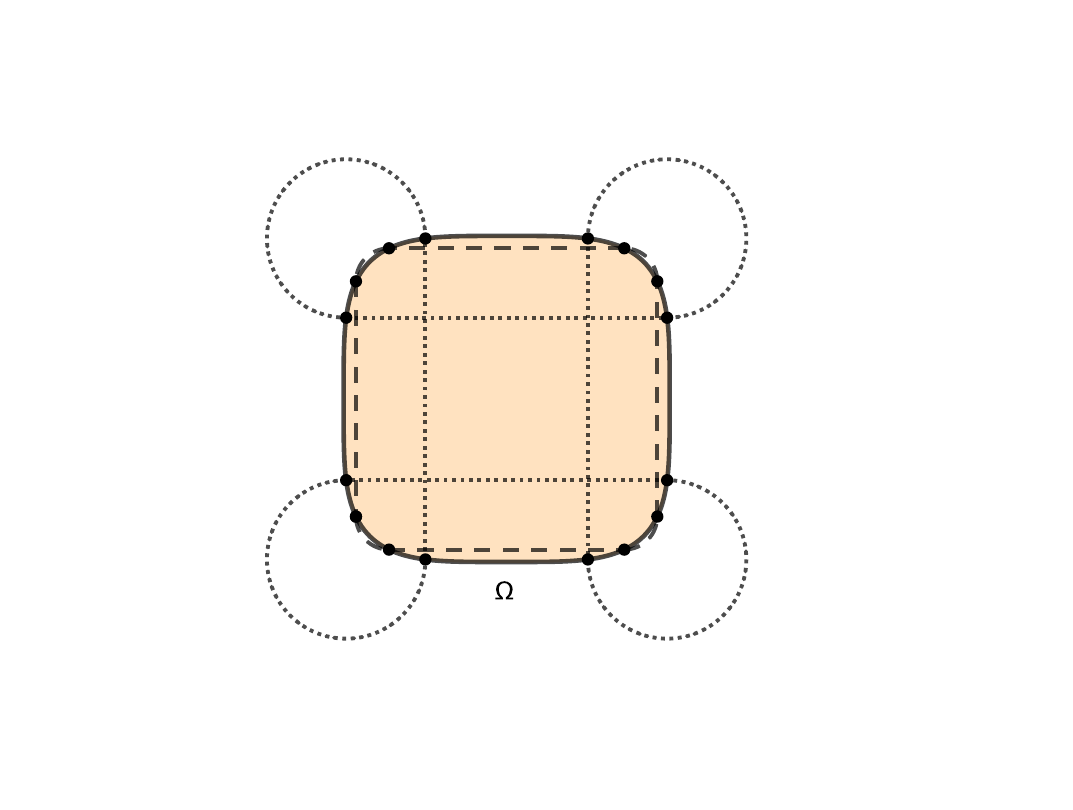}& (b) \includegraphics[width = 0.50\textwidth]{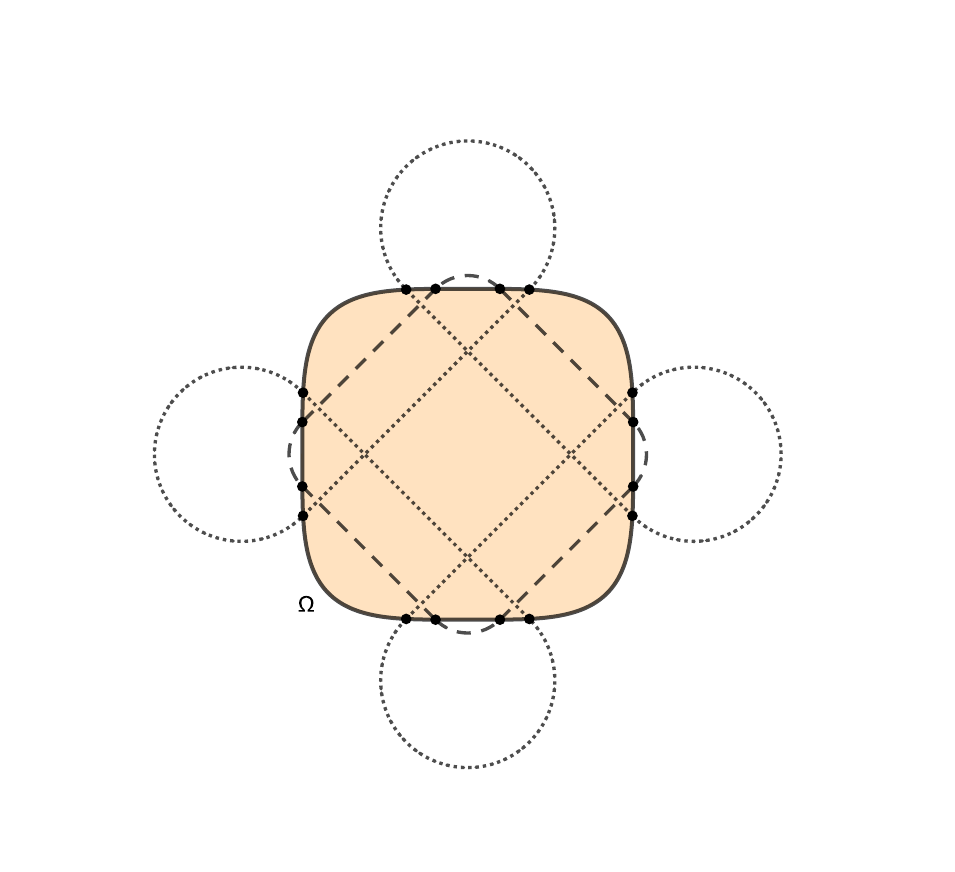} \\
\end{tabular}
\caption{Symmetric 4-periodic trajectories in the curve $x^{2k}+y^{2k}=1$ with rotation number $1/4$ (dashed) and $3/4$ (dotted), with different Larmor radii, and with Larmor centers on (a) the lines $y=\pm x$, and (b) on the coordinate axes. }
\label{4PeriodicSquarishDuality}
\end{figure}

Next, consider the case when the rotation number is $1/4$ and the Larmor centers are on the coordinate axes. Then, as before, all $\theta_i$, $\ell_{2i}$, $\ell_{2i+1}$, $\chi_{2i,2i+2}$ are constant. If the point $P_0 = (x_0,y_0)$ is in the first quadrant, then $x_0 \in (2^{-\frac{1}{2k}},1)$, $P_1 = (y_0,x_0)$, and $P_2 = (-y_0,x_0)$. Simple geometry yields
\begin{equation}
\begin{split}
\ell_1 &= \sqrt{2} (x_0-y_0) \\
\mu &= \sqrt{2} y_0 \\
\ell_2 &= 2y_0 \\
\tan(\theta) &= \frac{x_0^{2k}y_0-x_0 y_0^{2k}}{x_0^{2k}y_0+x_0 y_0^{2k}} 
\end{split}
\label{4PeriodicEquationsAgainAgain}
\end{equation} 
and the Larmor centers are at the points $(0,\pm(x_0-y_0))$ and $(\pm(x_0-y_0),0)$.

These quantities produce the following expression for $\Tr S_4$ in terms of $x_0$ and $y_0$:
\begin{align}
\begin{split}
\Tr S_4 &= 2 - \left[64x_0^{2k+1}y_0^{2k} \left(x_0^{2k}y_0^2 -x_0^2y_0^{2k}\right) \left(x_0^{2k}(x_0-2y_0) + 2x_0y_0^{2k}\right)  \right. \\
	&\;\;\;\;\; \cdot \left. \left(x_0^{4k}y_0^2-x_0^2 y_0^{4k} - 2(x_0-y_0)x_0^{2k+1}y_0^{2k}\right)^2  \right] / \left[ x_0^{2k}y_0 - x_0 y_0^{2k}\right]^8
\end{split}
\label{4PerTraceAgainAgain}
\end{align}

The above rational function equals zero at exactly one point, $x_0 = \check{x}$ in the interval $(2^{-\frac{1}{2k}},1)$. A quick analysis of this rational function and the above expression tells us the trajectory will be hyperbolic for $x_0 \in (2^{-\frac{1}{2k}},\check{x})$, parabolic for $x_0 = \check{x}$, and elliptic for $x_0 \in (\check{x},1)$.

We can repeat the above calculations in the case when the rotation number is $3/4$. In particular, this case is valid for $x_0 \in (-2^{-\frac{1}{2k}},1)$. Suppose $P_0 = (x_0,y_0)$ is in the first quadrant and $P_1 = (-y_0,x_0)$ is in the third quadrant. Then $P_2 = (y_0,-x_0)$ is in the fourth quadrant and the point $P_0$ and $P_1$ determine the start of a 4-periodic trajectory. In this case we have
\begin{equation}
\begin{split}
\ell_1 &= \sqrt{2} (x_0+y_0) \\
\mu &= \sqrt{2} y_0 \\
\ell_2 &= 2y_0 \\
\tan(\theta) &= \frac{x_0^{2k}y_0+x_0 y_0^{2k}}{x_0 y_0^{2k} - x_0^{2k}y_0} 
\end{split}
\label{4PeriodicEquationsAgainAgainAgain}
\end{equation} 
and the Larmor centers are at the points $(0,\pm(x_0+y_0))$ and $(\pm(x_0+y_0),0)$.

These quantities produce a similar expression for $\Tr S_4$ in terms of $x_0$ and $y_0$ as the previous case:
\begin{align}
\begin{split}
\Tr S_4 &= 2 - \left[64x_0^{2k}y_0^{2k-2} \left(x_0^{2k-2} -y_0^{2k-2}\right) \left(x_0^{2k-1}(x_0+2y_0) + y_0^{2k}\right)  \right. \\
	&\;\;\;\;\; \cdot \left. \left(x_0^{4k-2} - y_0^{4k-2} - 2(x_0+y_0)x_0^{2k-1}y_0^{2k-2}\right)^2  \right] / \left[ x_0^{2k-1} + y_0^{2k-1}\right]^8.
\end{split}
\label{4PerTraceAgainAgainAgain}
\end{align}
An analysis of this equation yields five values of $x_0 \in (-2^{-\frac{1}{2k}},1)$ for which $\Tr S_4$ is parabolic, call them $\tilde{x}_j$, $j \in \{1,\ldots, 5\}$ with $x_i < x_j$ for $i < j$. In particular $\tilde{x}_1 = 0$. Further analysis yields that the trajectory is 
\begin{itemize}
\item hyperbolic for $x_0 \in (-2^{-\frac{1}{2k}},\tilde{x}_1) \cup (\tilde{x}_1,\tilde{x}_2)$
\item parabolic for $x_0 \in \{\tilde{x}_1,\tilde{x}_2,\tilde{x}_3,\tilde{x}_4,\tilde{x}_5\}$
\item elliptic for $x_0 \in  (\tilde{x}_2, \tilde{x}_3) \cup (\tilde{x}_3,\tilde{x}_4) \cup (\tilde{x}_4,\tilde{x}_5)\cup (\tilde{x}_5,1)$. 
\end{itemize}

\end{example}

\section{Conclusions}
We have established a linear stability criterion for 2-periodic inverse magnetic billiard trajectories in convex and concave domains and have applied this to examples such as the ellipse and one-parameter family of curves given by $x^{2k}+y^{2k}=1$ for $k \in \N$. Further, comparisons have been made to the linear stability criteria for standard and magnetic billiards, noting the similarities (e.g. analogous stability results for billiards and inverse magnetic billiards in the ellipse) and differences (e.g. the linear stability of inverse magnetic billiards does not depend upon the curvature of the boundary at the exit and reentry points compared to the linear stability of standard billiards which does depend upon the curvature of the boundary at the reflection points). We also established the linear stability of symmetric 3- and 4-periodic inverse magnetic billiards trajectories in the same domains. Further, we made a geometric characterization of an isolated parabolic 2-periodic billiard trajectory in an ellipse in terms of rotation numbers and caustics. 

A further direction of research in this area is reconciling the following: the approach of \cite{KozlovTr} establishes linear stability criterion for billiards using the generating function $\mathcal{L}$ of the billiard map and its Hessian determinant evaluated at nondegenerate critical points. This criterion is identical to the criterion constructed via the stability matrix $S_n$. However, this generating function approach to linear stability in the setting of inverse magnetic billiards does not produce viable results in the 2-periodic case: the generating function $G$ for inverse magnetic billiards is not sufficiently differentiable when $\chi = \pi/2$. An equivalent statement in the magnetic billiards setting is made in \cite{BK} in the context of asymptotic expansions of the magnetic billiard map. The technique of \cite{KozlovTr} establishes a method to transform a degenerate critical point to a nondegenerate critical point, and finding magnetic and inverse magnetic analogues to this method constitute one path to overcoming this obstruction.

\section*{Acknowledgements}

This research is supported by the Discovery Project No.~DP190101838 \emph{Billiards within confocal quadrics and beyond} from the Australian Research Council. A portion of this research is based upon work supported by the National Science Foundation under Grant No. DMS-1440140 while the author was in residence at the Mathematical Sciences Research Institute in Berkeley, California, during the Fall 2018 semester.

\bibliographystyle{amsplain}
\nocite{*}
\bibliography{ReferencesStab}

\providecommand{\bysame}{\leavevmode\hbox to3em{\hrulefill}\thinspace}
\providecommand{\MR}{\relax\ifhmode\unskip\space\fi MR }
\providecommand{\MRhref}[2]{%
  \href{http://www.ams.org/mathscinet-getitem?mr=#1}{#2}
}
\providecommand{\href}[2]{#2}
\begin{thebibliography}{10}

\bibitem{BB}
V.~M. Babich and V.~S. Buldyrev, \emph{Asymptotic methods in short wave
  diffraction problems. vol. 1}, Izdat. ``Nauka'', Moscow, 1972, With the
  collaboration of M. M. Popov and I. A. Molotkov. (Russian). \MR{0426630}

\bibitem{BK}
N.~Berglund and H.~Kunz, \emph{Integrability and ergodicity of classical
  billiards in a magnetic field}, J. Statist. Phys. \textbf{83} (1996),
  no.~1-2, 81--126. \MR{1382763}

\bibitem{BMS2020}
Misha Bialy, Andrey~E. Mironov, and Lior Shalom, \emph{Magnetic billiards:
  Non-integrability for strong magnetic field; gutkin type examples}, Journal
  of Geometry and Physics \textbf{154} (2020), 103716.

\bibitem{Bir}
G.~Birkhoff, \emph{Dynamical {Systems}}, American {Mathematical} {Society} /
  {Providence}, {RI}, American Mathematical Society, 1927.

\bibitem{CRR2012}
Pablo~S. Casas and Rafael Ram\'{i}rez-Ros, \emph{Classification of symmetric
  periodic trajectories in ellipsoidal billiards}, Chaos: An Interdisciplinary
  Journal of Nonlinear Science \textbf{22} (2012), no.~2, 026110.

\bibitem{CP2012}
Giulio Casati and Toma\v{z} Prosen, \emph{Time {Irreversible} {Billiards} with
  {Piecewise}-{Straight} {Trajectories}}, Phys. Rev. Lett. \textbf{109} (2012),
  no.~17, 174101.

\bibitem{DR2019}
Vladimir Dragovi\'{c} and Milena Radnovi\'{c}, \emph{Periodic ellipsoidal
  billiard trajectories and extremal polynomials}, Communications in
  Mathematical Physics \textbf{372} (2019), no.~1, 183--211.

\bibitem{D}
J.J. Duistermaat, \emph{Discrete integrable systems}, Springer Monographs in
  Mathematics, vol. 304, Springer New York, 2010.

\bibitem{Dullin1998}
Holger~R. Dullin, \emph{Linear stability in billiards with potential},
  Nonlinearity \textbf{11} (1998), no.~1, 151--173. \MR{1492955}

\bibitem{G2019}
Sean Gasiorek, \emph{On the dynamics of inverse magnetic billiards}, Ph.D.
  thesis, University of California Santa Cruz, 2019, p.~90. \MR{4035388}

\bibitem{G2021}
\bysame, \emph{On the dynamics of inverse magnetic billiards}, Nonlinearity
  \textbf{34} (2021), no.~3, 1503--1524.

\bibitem{Gutkin2001}
Boris Gutkin, \emph{Hyperbolic magnetic billiards on surfaces of constant
  curvature}, Communications in Mathematical Physics \textbf{217} (2001),
  no.~1, 33--53.

\bibitem{KPC2005}
Bence Kocsis, Gergely Palla, and J\'{o}zsef Cserti, \emph{Quantum and
  semiclassical study of magnetic quantum dots}, Physical Review B \textbf{71}
  (2005), no.~7, 075331.

\bibitem{KROC2008}
A.~Korm\'anyos, P.~Rakyta, L.~Oroszl\'any, and J.~Cserti, \emph{Bound states in
  inhomogeneous magnetic field in graphene: {Semiclassical} approach}, Phys.
  Rev. B \textbf{78} (2008), no.~4, 045430.

\bibitem{KozlovTr}
V.~V. Kozlov and D.~V. Treshch\"{e}v, \emph{Billiards}, Translations of
  Mathematical Monographs, vol.~89, American Mathematical Society, Providence,
  RI, 1991, A genetic introduction to the dynamics of systems with impacts,
  Translated from the Russian by J. R. Schulenberger. \MR{1118378}

\bibitem{K2000}
V.V. Kozlov, \emph{Two-link billiard trajectories: extremal properties and
  stability}, Journal of Applied Mathematics and Mechanics \textbf{64} (2000),
  no.~6, 903--907.

\bibitem{KP05}
V.V. Kozlov and S.A. Polikarpov, \emph{Periodic billiard trajectories in a
  magnetic field}, Journal of Applied Mathematics and Mechanics \textbf{69}
  (2005), no.~6, 844--851.

\bibitem{RB1985}
M~Robnik and M~V Berry, \emph{Classical billiards in magnetic fields}, Journal
  of Physics A: Mathematical and General \textbf{18} (1985), no.~9, 1361--1378.

\bibitem{Robnik1986}
Marko Robnik, \emph{Regular and chaotic billiard dynamics in magnetic fields},
  Nonlinear Phenomena and Chaos, Bristol, 1986.

\bibitem{Tab}
S.~Tabachnikov, \emph{Geometry and {Billiards}}, Student mathematical library,
  American Mathematical Society, 2005.

\bibitem{TabMag}
Serge Tabachnikov, \emph{Remarks on magnetic flows and magnetic billiards,
  {F}insler metrics and a magnetic analog of {H}ilbert's fourth problem},
  Modern dynamical systems and applications (2004), 233--250.

\bibitem{VTCP}
Z.~V{\"o}r{\"o}s, T.~Tasn{\'a}di, J.~Cserti, and P.~Pollner, \emph{Tunable
  {Lyapunov} exponent in inverse magnetic billiards}, Physical Review E
  \textbf{67} (2003), no.~6, 065202.

\bibitem{W1986}
Maciej Wojtkowski, \emph{Principles for the design of billiards with
  nonvanishing lyapunov exponents}, Communications in Mathematical Physics
  \textbf{105} (1986), no.~3, 391--414.

\end{thebibliography}

\smallskip 

\hrule

\smallskip

\end{document}